\documentclass[a4paper, 11pt, reqno]{amsart}
\usepackage[english]{babel}
\usepackage[utf8]{inputenc}
\usepackage[T1]{fontenc}
\usepackage{amsthm, amsmath, amssymb, amsrefs, enumerate, enumitem, mathtools, microtype, textgreek, lmodern, fullpage, xurl}
\usepackage{hyperref}

\numberwithin{equation}{section}
\setlist{nosep}
\setlist{noitemsep}

\theoremstyle{plain}
\newtheorem{thm}{Theorem}[section]
\newtheorem{prop}[thm]{Proposition}
\newtheorem{lemma}[thm]{Lemma}
\newtheorem{cor}[thm]{Corollary}
\newtheorem*{rmk}{{\bf Remark}}
\newtheorem*{rmks}{{\bf Remarks}}

\theoremstyle{definition}
\newtheorem{defn}[thm]{Definition}

\newcommand{\C}{\mathbb{C}}
\newcommand{\D}{\mathbb{D}}
\renewcommand{\H}{\mathbb{H}}
\newcommand{\Z}{\mathbb{Z}}
\newcommand{\Q}{\mathbb{Q}}
\newcommand{\N}{\mathbb{N}}
\newcommand{\R}{\mathbb{R}}
\newcommand{\re}{\operatorname{Re}}
\newcommand{\im}{\operatorname{Im}}
\newcommand{\Arg}{\operatorname{Arg}}
\newcommand{\SL}{\operatorname{SL}}

\newcommand{\vt}[1]{\left\lvert #1 \right\rvert}
\newcommand{\Ec}{\mathcal{E}}
\newcommand{\Qc}{\mathcal{Q}}

\DeclareMathOperator{\sgn}{sgn}
\DeclareMathOperator{\Log}{Log}
\newcommand{\ol}[1]{\overline{#1}}

\renewcommand{\a}{\alpha}
\renewcommand{\b}{\beta}
\newcommand{\de}{\delta}
\newcommand{\del}{\partial}
\newcommand{\eps}{\varepsilon}
\newcommand{\g}{\gamma}
\renewcommand{\k}{\kappa}
\renewcommand{\l}{\lambda}
\newcommand{\om}{\omega}
\renewcommand{\t}{\tau}
\newcommand{\De}{\Delta}
\newcommand{\Ga}{\Gamma}
\newcommand{\La}{\Lambda}
\newcommand{\Om}{\Omega}

\newcommand{\Qf}{\mathfrak{Q}}
\newcommand{\ff}{\mathfrak{f}}

\newcommand{\mat}[1]{\begin{pmatrix}#1\end{pmatrix}}
\newcommand{\pmat}[1]{\left(\begin{smallmatrix}#1\end{smallmatrix}\right)}
\newcommand{\abcd}{\begin{pmatrix}a&b\\c&d\end{pmatrix}}
\newcommand{\pabcd}{\left(\begin{smallmatrix}a&b\\c&d\end{smallmatrix}\right)}

\renewcommand{\pmod}[1]{\   \left(  \mathrm{mod} \,  #1 \right)}
\newcommand{\Pmod}[1]{\   (  \mathrm{mod} \,  #1 )}

\title[A modular framework for functions of Knopp]{A modular framework for functions of Knopp and indefinite binary quadratic forms}
\author{Kathrin Bringmann}
\address{Department of Mathematics and Computer Science, Division of Mathematics, University of Cologne, Weyertal 86--90, 50931 Cologne, Germany}
\email{kbringma@math.uni-koeln.de}
\author{Andreas Mono}
\address{Department of Mathematics and Computer Science, Division of Mathematics, University of Cologne, Weyertal 86--90, 50931 Cologne, Germany}
\curraddr{\textsc{Department of Mathematics, Vanderbilt University, 1326 Stevenson Center, Nashville, TN 37240, USA}}
\email{andreas.mono@vanderbilt.edu}
\date{\today}

\begin{document}

\begin{abstract}
We study functions introduced by Knopp and complete them to non-holomorphic bimodular forms of positive integral weight related to indefinite binary quadratic forms. We investigate further properties of our completions, which in turn motivates certain local cusp forms. We then define modular analogues of negative weight of our local cusp forms, which are locally harmonic Maass forms with continuously removable singularities. We show that they admit local splittings in terms of Eichler integrals.
\end{abstract}


\maketitle

\section{Introduction and statement of results}

Throughout the paper $D>0$ is a non-square discriminant, $k\in2\N$, $\Qc_d$ denotes the set of integral binary quadratic forms $Q=[a,b,c]$ of discriminant $d\in\Z$, and $\H$ is the complex upper half-plane. In $1975$, Zagier \cite{zagier75} introduced the functions\footnote{We define $f_{\k,D}$ in Zagier's original normalization, which differs from the normalization used in \cite{BKK}.}
\[
	f_{\k,D}(\t) \coloneqq \sum_{Q\in\Qc_D} \frac{1}{Q(\t,1)^\k},\qquad \t \in \H,
\]
and proved that they are weight $2\kappa$ cusp forms if $\k>1$ (if $\k=1$, one may use Hecke's trick, see \cite{koh}*{p.\ $239$}). To name a few prominent applications of the $f_{\k,D}$, they are coefficients of the holomorphic kernel function of the Shimura \cite{shim} and Shintani \cite{shin} lifts due to \cite{koh}, and they are closely related to central $L$-values by \cite{koza81}. Their even periods are rational according to \cite{koza84}, and they generate the space of weight $2\kappa$ cusp forms \cite{katok}.

Over $30$ years ago, Knopp \cite{knopp90}*{$(4.5)$} found a term-by-term preimage of each $f_{k,D}$ under the {\it Bol operator} $\D^{2k-1}$, where $\D\coloneqq\frac{1}{2\pi i}\frac{\del}{\del\t}$ (compare Proposition \ref{prop:psiprop} (2)). We refer to Knopp's earlier work \cites{knopp62, knopp64, knopp74} and to the references \cites{bruonrho, eichler, lehner71} for the importance of the Bol operator. To ensure convergence after summing over $Q\in\Qc_D$, Knopp changed the sign of $k$ in his result afterwards, which lead to (throughout $\Log$ denotes the principal branch of the complex logarithm)
\begin{equation}\label{eq:psidef}
	\psi_{k+1,D}(\t) \coloneqq \sum_{Q\in\Qc_D} \frac{\Log\left(\frac{\t-\a_Q^-}{\t-\a_Q^+}\right)}{Q(\t,1)^{k+1}},\qquad \a_{[a,b,c]}^\pm \coloneqq \frac{-b\pm\sqrt{D}}{2a} \in \R.
\end{equation}
He also stated that $\psi_{k+1,D}(\t+1)=\psi_{k+1,D}(\t)$, and the behaviour of $\psi_{k+1,D}$ under modular inversion\footnote{We alert the reader to the fact that Knopp used the older convention $T=\pmat{0&-1\\1&0}$.} (see \cite{knopp90}*{(4.6)}). Correcting a typo there, we find that (see Proposition \ref{prop:psiprop} (3))
\begin{equation}\label{eq:psitransf}
	\t^{-2k-2}\psi_{k+1,D}\left(-\frac1\t\right) - \psi_{k+1,D}(\t) = \sum_{Q\in\Qc_D} \frac{\log\vt{\frac{\a_Q^+}{\a_Q^-}}}{Q(\t,1)^{k+1}} - 2\pi i\sum_{\substack{Q=[a,b,c]\in\Qc_D\\a<0<c}} \frac{1}{Q(\t,1)^{k+1}}.
\end{equation}

On the one hand, we observe that $\psi_{k+1,D}$ is holomorphic and vanishes at $i\infty$ (this follows by Proposition \ref{prop:psiprop} (1) and \eqref{eq:loglimit}). On the other hand, $\psi_{k+1,D}$ itself is not modular. Hence, it is natural to ``complete'' $\psi_{k+1,D}$. Setting
$
\H^-\coloneqq -\H
$
throughout, completions of $\psi_{k+1,D}$ are bimodular forms\footnote{We slightly modify the initial definition by Stienstra and Zagier \cite{stza} to include the domain $\H\times\H^-$.} $\Om_{k+1,D}$ of weight $(2k+2,0)$ defined on $\H\times\H^-$ such that
\begin{equation} \label{eq:naturality}
	\lim_{w\to -i\infty} \Omega_{k+1,D}(\t,w) = \psi_{k+1,D}(\t).
\end{equation}
Thus, from the completions $\Omega_{k+1,D}$ one can uniquely recover the original functions $\psi_{k+1,D}$.

In this paper, we construct such completions explicitly. Firstly, we note that the final sum appearing in \eqref{eq:psitransf} is finite, because $b^2+4\vt{ac}=D>0$ has only finitely many integral solutions. This leads to Knopp's modular integrals with rational period functions \cites{knopp78, knopp81}. Roughly speaking, period polynomials describe the obstruction of modularity of Eichler integrals \cite{eichler} (defined in \eqref{eq:Eichlerdef}) of cusp forms, and Knopp generalized this notion to rational functions instead of polynomials. Such functions are called {\it modular integrals}. Parson \cite{parson} constructed such modular integrals explicitly by letting
\begin{equation} \label{eq:phidef}
	\varphi_{k+1,D}(\tau) \coloneqq \frac{1}{2}\sum_{Q \in \Qc_D} \frac{\sgn(Q)}{Q(\tau,1)^{k+1}} = \sum_{\substack{Q = [a,b,c]  \in \Qc_D\\a>0}} \frac{1}{Q(\tau,1)^{k+1}}, \qquad \sgn\left([a,b,c]\right) \coloneqq \sgn(a) 
\end{equation}
and we recall her result on the $\varphi_{k+1,D}$ in Lemma \ref{lem:phitransf}. Secondly, we define
\begin{equation*}
	Q_w \coloneqq \frac{1}{\im(w)}\left(a\vt{w}^2+b\re(w)+c\right),\qquad S_Q \coloneqq \{\t\in\H \colon Q_\t=0\},\qquad E_D \coloneqq \bigcup_{Q\in\Qc_D} S_Q,
\end{equation*}
for $w\in\C\setminus\R$, $Q \in \Qc_d$ ($d \in \Z$), as well as the functions
\begin{equation}\label{eq:rhodef}
	\rho_{k+1,D}(\t,w) \coloneqq \sum_{Q\in\Qc_D} \frac{\Log\left(\frac{w-\a_Q^-}{w-\a_Q^+}\right)}{Q(\t,1)^{k+1}},\qquad \l_{k+1,D}(\t,w) \coloneqq 2i\sum_{Q\in\Qc_D} \frac{\arctan\left(\frac{Q_w}{\sqrt{D}}\right)}{Q(\t,1)^{k+1}}. 
\end{equation}
for $w\in\H^-$. We refer to Propositions \ref{prop:rhoprop} and \ref{prop:lambdaprop} for some of their properties. 

Thirdly, we define completions of $\psi_{k+1,D}$ as\footnote{The name completion is justified by Theorem \ref{thm:OmegaMmain} (2).}
\begin{align}
	\Omega_{k+1,D}(\t,w) &\coloneqq \psi_{k+1,D}(\t)-\rho_{k+1,D}(\tau,w) + 2\pi i \varphi_{k+1,D}(\tau) + \lambda_{k+1,D}(\tau,w), \label{eq:OmegaMdef}
\end{align}
on $\H\times\H^-$. In Proposition \ref{prop:lambdaprop} (2), we show that the functions $\lambda_{k+1,D}$ are theirself bimodular of the same weights as the $\Omega_{k+1,D}$. Their purpose is to ensure parts (2) to (5) of the following theorem (compare Proposition \ref{prop:lambdaprop} (3) and \eqref{eq:VanishingOnDiagonal}). In particular, the functions $\lambda_{k+1,D}$ ensure that the $\Omega_{k+1,D}$ from \eqref{eq:OmegaMdef} satisfy \eqref{eq:naturality}.

\begin{thm} \label{thm:OmegaMmain}
	Let $\tau \in \H$, $w \in \H^-$.
	\begin{enumerate}[leftmargin=*,label=\rm{(\arabic*)}]
		\item The functions $\Om_{k+1,D}$ are bimodular of weight $(2k+2,0)$ that is
		\[
			\Omega_{k+1,D}(\t+1,w+1) = \Omega_{k+1,D}(\t,w),\qquad \Omega_{k+1,D}\left(-\frac1\t,-\frac1w\right) = \t^{2k+2}\Omega_{k+1,D}(\t,w).
		\]

		\item We have
		\begin{equation*}
			\lim_{w\to -i\infty} \Omega_{k+1,D}(\t,w) = \psi_{k+1,D}(\t).
		\end{equation*}
		
		\item We have
		\begin{align*}
			\lim_{\tau \to i\infty} \Omega_{k+1,D}(\tau,w) = 0.
		\end{align*}

		\item The functions $\Omega_{k+1,D}$ are holomorphic with respect to $\tau$ and anti-holomorphic with respect to $w$.

		\item We have that
		\begin{align*}
			\Omega_{k+1,D}(\tau,\ol\tau) = 0.
		\end{align*}
	\end{enumerate}
\end{thm}

\begin{rmk}
In \cite{DIT10}*{Theorem 3}, Duke, Imamo\={g}lu, and T\'{o}th constructed (weakly) holomorphic modular integrals $F$ with rational period functions. It turns out that Parson's modular integral $\varphi_{k+1,D}$ has the same rational period functions as the ones arising from $F$, see Lemma \ref{lem:phitransf}. In \cite{DIT10}*{(16)}, Duke, Imamo\={g}lu, and T\'{o}th showed that both modular integrals differ by a (non-explicit) cusp form of weight greater than $2$. For weight $2$, the second author \cite{mo22}*{Theorem 1.1} showed that both modular integrals differ by an explicit non-zero multiple of the weight $2$ non-holomorphic Eisenstein series $\widehat{E}_2$. Theorem \ref{thm:OmegaMmain} embeds Parson's modular integral $\varphi_{k+1,D}$ into a non-holomorphic bimodular framework in higher weights.
\end{rmk}

Theorem \ref{thm:OmegaMmain} explains how modular integrals may exhibit modular properties by completing them to non-holomorphic bimodular forms. It turns out that the $\Omega_{k+1,D}$ can be extended to $\H \times \H$, and become holomorphic bimodular forms with analogous properties there. Being more precise, we define the functions
	\[
		\om_{k+1,D}(\t,z) \coloneqq \psi_{k+1,D}(\t) - \sum_{Q\in\Qc_D} \frac{\Log\left(\frac{z-\a_Q^-}{z-\a_Q^+}\right)}{Q(\t,1)^{k+1}},\qquad \t,z \in \H,
	\]
and obtain the following corollary.
\begin{cor}\label{cor:omegacor}
Let $\t,z \in \H$. We have
\begin{align*}
\om_{k+1,D}(\t,z) = \Omega_{k+1,D}\left(\tau,\overline{z}\right).
\end{align*}
In other words, the $\om_{k+1,D}$ satisfy
	\begin{align*}
		\om_{k+1,D}(\t+1,z+1) &= \omega_{k+1,D}(\t,z),\qquad \omega_{k+1,D}\left(-\frac1\t,-\frac1z\right) = \t^{2k+2}\omega_{k+1,D}(\t,z), \\
		\lim_{z \to i\infty} \omega_{k+1,D}(\tau,z) &= \psi_{k+1,D}(\tau), \qquad \lim_{\tau \to i\infty} \omega_{k+1,D}(\tau,z) = 0, \\
		\omega_{k+1,D}(\tau,\tau) &= 0,
	\end{align*}
	and are holomorphic with respect to $\t$ and $z$.
\end{cor}

To prove Theorem \ref{thm:OmegaMmain} (5), we specialize the $\lambda_{k+1,D}$ to $(\tau,w) = (\tau,\ol\tau) \in \H \times \H^-$. Roughly speaking, \eqref{eq:Lambdatemp} demonstrates that they compensate for the singularities of
\begin{align*}
\Omega_{k+1,D}\left(\t,\ol\t\right) - \lambda_{k+1,D}\left(\t,\ol\t\right) = \sum_{Q \in \Qc_D} \frac{\Log\left(\frac{\frac{Q_\t}{\sqrt{D}}-i}{\frac{Q_\t}{\sqrt{D}}+i}\right)} {Q(\tau,1)^{k+1}} + \pi i\sum_{\Qc\in\Qc_D} \frac{\sgn\left(Q_{\tau}\right)}{Q(\t,1)^{k+1}}
\end{align*}
on the set $E_D$. The functions 
\begin{align}
	\La_{k+1,D}(\t) \coloneqq \sum_{Q\in\Qc_D} \frac{\sgn(Q_\t)}{Q(\t,1)^{k+1}} \label{eq:Lambdadef}
\end{align}
appeared first in a paper \cite{mo21} by the second author, and turn out to be {\it local cusp forms}. That is, they behave like cusp forms of weight $2k+2$ outside $E_D$, however, in addition, have jumping singularities\footnote{We explain this terminology in Section \ref{sec:prelim}.} on $E_D$. A full definition of such functions can be found in \cite{mo21}*{Definition 2.7}, which in turn adapts an earlier definition by the first author, Kane, and Kohnen \cite{BKK}*{Section 2}, see Proposition \ref{prop:Lambdacusp} for more details as well. By \cite{mo21}*{Theorem 1.1}, the functions $\La_{k+1,D}$ can be written in terms of traces of cycle integrals. Alternatively, the $\La_{k+1,D}$ might be viewed ``odd'' positive weight analogues of the $f_{k,D}$. Recently, the $f_{k,D}$ motivated the introduction of new modular objects $\mathcal{F}_{1-k,D}$ by the first author, Kane, and Kohnen \cite{BKK}. The function $\mathcal{F}_{1-k,D}$ maps to $f_{k,D}$ under the Bol operator as well as the {\it shadow operator} $\xi_{2-2k}\coloneqq2iv^{2-2k}\ol{\frac{\del}{\del\ol\t}}$ of Bruinier and Funke \cite{brufu} (up to constants). Such a behaviour is impossible in the situation of a (globally defined) non-trivial harmonic Maass form\footnote{One may overcome this by weakening the growth condition in Definition \ref{defn:autobj}, see \cite{thebook}*{Theorem 6.15}.}. Hence, it is natural to construct ``even'' analogues $\Psi_{-k,D}$ of the $\mathcal{F}_{1-k,D}$ along the lines of \cite{BKK}. For this, we let $\b(x;s,w)\coloneqq\int_0^xt^{s-1}(1-t)^{w-1}dt$, $x\in(0,1]$, $\re(s)$, $\re(w)>0$, be the {\it incomplete $\b$-function}, $\t=u+iv$ throughout, and
\[
	\Psi_{-k,D}(\tau) \coloneqq \frac{1}{2}\sum_{Q \in \Qc_D} Q(\tau,1)^k \beta\left(\frac{Dv^2}{\vt{Q(\tau,1)}^2};k+\frac{1}{2},\frac{1}{2}\right), \qquad \tau \in \H \setminus E_D.
\]
In the spirit of Knopp's initial preimage of $f_{k,D}$ under the Bol operator (without an additional sign change of $k$), it turns out that $\Psi_{-k,D}$ is a preimage of $\La_{k+1,D}$ under the Bol operator and the shadow operator. If $f$ is a cusp form of weight $2k+2$, then preimages under $\D^{2k+1}$ and $\xi_{-2k}$, respectively, are provided by the holomorphic and non-holomorphic {\it Eichler integrals} (see \eqref{eq:Eichlerdiff})
\begin{equation} \label{eq:Eichlerdef}
	\Ec_f(\t) \coloneqq -\frac{(2\pi i)^{2k+1}}{(2k)!}\!\int_\t^{i\infty}\! f(w)(\t-w)^{2k} dw, \quad 
	f^*(\t) \coloneqq (2i)^{-2k-1}\!\int_{-\ol\t}^{i\infty} \ol{f\left(-\ol w\right)}(w+\t)^{2k} dw.
\end{equation}
To be able to insert the local cusp forms $\La_{k+1,D}$ into each integral in \eqref{eq:Eichlerdef}, we work in a fundamental domain of $\SL_2(\Z)$, in which we have just finitely many equivalence classes of geodesics $S_Q$. Integrating piecewise, both Eichler integrals of $\La_{k+1,D}$ are well-defined on $\H\setminus E_D$. In addition we ensure in Proposition \ref{prop:Eichlerlocal} that both Eichler integrals of $\La_{k+1,D}$ exist on $E_D$. This established, we prove the following properties of $\Psi_{-k,D}$. We refer the reader to Subsection \ref{sec:LHMF} for definitions.

\begin{thm}\label{thm:Psimain}
	\ \begin{enumerate}[leftmargin=*,label=\rm{(\arabic*)}]
		\item The functions $\Psi_{-k,D}$ are locally harmonic Maass forms of weight $-2k$ with continuously (however not differentially) removable singularities on $E_D$. 

		\item If $\t\in\H\setminus E_D$, then we have
		\[
			\Psi_{-k,D}(\t) = c_\infty - \frac{D^{k+\frac12}(2k)!}{(4\pi)^{2k+1}}\Ec_{\La_{k+1,D}}(\t) + D^{k+\frac12}\La_{k+1,D}^*(\t),
		\]
		where
		\begin{equation*}
			c_\infty \coloneqq \frac{\pi D^{k+\frac12}}{2^{2k}(2k+1)}\sum_{a\ge1} \sum_{\substack{0\le b<2a\\b^2\equiv D\pmod{4a}}} \frac{1}{a^{k+1}}.
		\end{equation*}
	\end{enumerate}
\end{thm}

\begin{rmks}
	\ 
\begin{enumerate}[leftmargin=*,label=\rm{(\arabic*)}]
\item Using a different normalization, the constant $c_{\infty}$ was introduced in \cite{BKK}*{(4.2), (7.3)}, and can be evaluated using a result of Zagier \cite{zagier77}*{Proposition 3}.
\item We prove Theorem \ref{thm:Psimain} {\rm (2)} by induction on $k$ in Section \ref{sec:Psisec}, while the first author, Kane and Kohnen utilize hyperbolic expansions\footnote{An excellent survey on such expansions can be found in \cite{IOS} for example.} to prove their corresponding result \cite{BKK}*{Theorem 1.3}.
\item In joint work with Rolen \cite{BMR}, the authors showed that both $\mathcal{F}_{1-k,D}$ (resp.\ $\Psi_{-k,D}$) map to $-\mathcal{F}_{1-k,D}$ (resp.\ $-\Psi_{-k,D}$) under the so-called flipping operator. 
\end{enumerate}
\end{rmks}

The paper is organized as follows. We recall results required for this paper in Section \ref{sec:prelim}. Section \ref{sec:Omegasec} is devoted to the proof of Knopp's initial claims on $\psi_{k+1,D}$, to some properties of the functions $\rho_{k+1,D}$, $\varphi_{k+1,D}$, $\lambda_{k+1,D}$, and to the proofs of Theorem \ref{thm:OmegaMmain} as well as of Corollary \ref{cor:omegacor}. In Section \ref{sec:Lambdasec} we investigate the behaviour of $\La_{k+1,D}$, $\Ec_{\La_{k+1,D}}$, and $\La_{k+1,D}^*$ on $E_D$. Section \ref{sec:Psisec} discusses the properties of $\Psi_{-k,D}$, and proves Theorem \ref{thm:Psimain}.

\section*{Acknowledgements}

The authors would like to thank Caner Nazaroglu for very helpful comments and discussions. Moreover we thank the referee for very helpful comments on our manuscript. Both authors were supported by the CRC/TRR 191 ``Symplectic Structures in Geometry, Algebra and Dynamics'', funded by the DFG (project number 281071066).

\section{Preliminaries}\label{sec:prelim}

\subsection{Integral binary quadratic forms and Heegner geodesics}

The modular group $\SL_2(\Z)$ acts on $\Qc_d$ by ($\pmat{a&b\\c&d}\in\SL_2(\Z)$)
\begin{equation*}
	\left(Q\circ\abcd\right)(x,y) \coloneqq Q(ax+by,cx+dy).
\end{equation*}
The action of $\SL_2(\Z)$ on $\H$ is compatible with the action of $\SL_2(\Z)$ on $\Qc_d$, in the sense that\footnote{A good reference is for example Zagier's book \cite{zagier81}*{§8}.}
\begin{equation} \label{eq:GammaOnQ}
	\left(Q \circ \gamma\right)(\tau,1) = j(\gamma,\tau)^2 Q(\gamma\tau,1), \qquad j\left(\left(\begin{matrix} a & b \\ c & d \end{matrix}\right),\tau\right) \coloneqq c\tau+d.
\end{equation}

Since $D>0$ is not a square, the two roots $\a_Q^\pm$ of $Q\in\Qc_D$ are real-quadratic and connected by the Heegner geodesic $S_Q$. We orientate $S_Q$ counterclockwise (resp.\ clockwise) if $\sgn(Q)>0$ (resp.\ $\sgn(Q)<0$). The orientation of $S_Q$ in turn determines the sign one catches if $\t$ jumps across $S_Q$. More precisely, one has $\sgn(Q)\sgn(Q_\t)<0$ if and only if $\t$ lies in the bounded component of $\H\setminus S_Q$. The unbounded connected component of $\H\setminus E_D$ is the unique such component containing $i\infty$ on its boundary. We refer the reader to the beautiful article by Duke, Imamo\={g}lu, and T\'{o}th \cite{DIT11}*{Section 4} for more on Heegner geodesics. 

We next collect results, which we utilize throughout. The following lemma is straightforward.

\begin{lemma} \label{lem:bkkidentity}
	For $Q \in \Qc_d$, $d \in \Z$, we have
	\[
		dv^2+Q_{\tau}^2v^2 = \vt{Q(\tau,1)}^2.
	\]
\end{lemma}

To determine the weights of our functions, the following lemma is useful.

\begin{lemma}\label{lem:InvarFactors}
	For every $Q\in\Qc_D$ and $\g\in\SL_2(\Z)$, we have
	\[
		(Q\circ\g)_\t = Q_{\g\t},\qquad \frac{\im(\g\t)}{\vt{Q(\g\t,1)}} = \frac{v}{\vt{(Q\circ\g)(\t,1)}}.
	\]
\end{lemma}

We also require the following elementary lemma. 

\begin{lemma}\label{lem:difftrick}
Let $U \subseteq \C$ be open.
	Assume that $f\colon U\to\C$ is real-differentiable and satisfies $\ol{f(\ol\t)}=f(\t)$. Then
	\[
		\ol{\frac{\del}{\del\ol\t} f\left(\ol\t\right)} = \frac{\del}{\del\t} f(\t).
	\]
\end{lemma}

The following differentiation rules are obtained by a direct calculation.

\begin{lemma}\label{lem:holomzoo}
Let $Q \in \Qc_D$.
\begin{enumerate}[leftmargin=*, label=\rm{(\arabic*)}]
\item We have
\begin{align*}
v^2 \frac{\partial}{\partial\tau} Q_{-\overline{\tau}} = \frac{i}{2}Q(-\overline{\tau},1), \qquad v^2 \frac{\partial}{\partial\tau}
 Q_{\tau} = \frac{i}{2} Q(\overline{\tau},1), \qquad v^2 \frac{\partial}{\partial\tau} \frac{Q(\tau,1)}{v^2} = i Q_{\tau}.
\end{align*}
\item We have
\begin{align*}
\frac{\partial}{\partial\overline{\tau}} \frac{v^2}{Q(\overline{\tau},1)} = \frac{iv^2Q_{\tau}}{Q(\overline{\tau},1)^2}, \qquad 2iv^2 \frac{\partial}{\partial\overline{\tau}} Q_{\tau} = Q(\tau,1), \qquad iv^2 \frac{\partial}{\partial\overline{\tau}} \frac{Q(\overline{\tau},1)}{v^2} = Q_{\tau}.
\end{align*}
\end{enumerate}
\end{lemma}

Letting $Q'(\t,1)\coloneqq\frac{\del}{\del\t}Q(\t,1)$, the following lemma can be verified by direct calculation.

\begin{lemma} \label{lem:Qtricks}
Let $Q\in\Qc_D$ and  $\t\in\H$. We have 
		\[
			Q_\t v + ivQ'(\t,1) = Q(\t,1),
		\qquad
			Q'(\t,1)^2 - 2Q''(\t,1)Q(\t,1) = D.
		\]
\end{lemma}

\subsection{Maass forms and modular forms}

Let $\k\in\frac12\Z$ and $d$ odd. Define
\begin{align*}
N \coloneqq \begin{cases}
1 & \text{if } \k\in\Z, \\
4 & \text{if } \k\in\Z+\frac12,
\end{cases}
\qquad
\eps_d\coloneqq \begin{cases}
1 & \text{if } d\equiv1\Pmod{4}, \\
i & \text{if } d\equiv3\Pmod{4},
\end{cases}
\end{align*}
Let $\left(\frac cd\right)$ be the extended Legendre symbol and $\pabcd\in\Ga_0(N)$. The {\it slash operator} is defined as
\[
	f\big|_\kappa\mat{a&b\\c&d}(\tau) \coloneqq
	\begin{cases}
		(c\t+d)^{-\k}f(\g\t) & \text{if } \k\in\Z,\\
		\left(\frac cd\right)\eps_d^{2\k}(c\t+d)^{-\k}f(\g\t) & \text{if } \k\in\frac12+\Z,
	\end{cases}
\]
The {\it weight $\k$ hyperbolic Laplace operator} is given as
\begin{equation*}
	\De_{\k} \coloneqq -v^2\left(\frac{\del^2}{\del u^2}+\frac{\del^2}{\del v^2}\right) + i\k v\left(\frac{\del}{\del u}+i\frac{\del}{\del v}\right).
\end{equation*}

We require various classes of modular objects. 
\begin{defn} \label{defn:autobj}
	Let $f\colon \H\to\C$ be a real-analytic function.
	\begin{enumerate}[leftmargin=*, label=(\arabic*)]
		\item We call $f$ a {\it (holomorphic) modular form} of weight $\kappa$ for $\Gamma_0(N)$, if $f$ satisfies the following:
		\begin{enumerate}[leftmargin=*,label=(\roman*), wide, labelwidth=0pt, labelindent=0pt]
		\item We have $f|_\k\g=f$ for all $\g\in\Ga_0(N)$.
		\item The function $f$ is holomorphic on $\H$.
		\item The function $f$ is holomorphic at the cusps of $\Ga_0(N)$.
		\end{enumerate}
		\item We call $f$ a {\it cusp form} of weight $\k$ for $\Ga_0(N)$, if $f$ is a modular form that vanishes at all cusps of $\Ga_0(N)$.
		
		\item We call $f$ a {\it harmonic Maass form of weight $\k$} for $\Ga_0(N)$, if $f$ satisfies the following:
		\begin{enumerate}[leftmargin=*,label=(\roman*), wide, labelwidth=0pt, labelindent=0pt]
		\item For every $\g\in\Ga_0(N)$ and every $\t\in\H$ we have that $f|_\k\g=f$.
		\item The function $f$ has eigenvalue $0$ under $\De_{\k}$.
		\item There exists a polynomial $P_f\in\C[q^{-1}]$ (the principal part of $f$) such that
		\[
			f(\t) - P_f(\t) = O\left({e}^{-\delta v}\right)
		\]
		\indent\indent as $v\to\infty$ for some $\de>0$, and we require analogous conditions at all other cusps of \indent\indent $\Ga_0(N)$.
		\end{enumerate}
	\end{enumerate}
Forms in {\it Kohnen's plus space} have the additional property that their Fourier expansion is supported on indices $n$ satisfying $(-1)^{\kappa-\frac{1}{2}} n \equiv 0$, $1 \pmod {4}$ with $\kappa \in \Z + \frac{1}{2}$.
\end{defn}

We remark that $\De_{\k}$ splits as
\begin{equation}\label{eq:split}
	\De_{\k} = -\xi_{2-\kappa} \circ \xi_{\kappa},
\end{equation}
which in turn implies that a harmonic Maass form naturally splits into a holomorphic and a non-holomorphic part. The operator $\xi_\k$ annihilates the holomorphic part, while the Bol operator $\D^{1-\k}$, $\k\in-\N_0$, annihilates the non-holomorphic part (since our growth condition rules out a non-holomorphic constant term in the Fourier expansion). Letting $\ell\in\N$, the Bol operator can be written in terms of the {\it iterated Maass raising operator} 
\begin{align} \label{eq:bolidentity}
	(-4\pi)^{\ell-1} \D^{\ell-1} = R_{2-\ell}^{\ell-1} \coloneqq R_{\ell-2} \circ \ldots \circ R_{2-\ell+2} \circ R_{2-\ell}, \quad R_{2-\ell}^0 \coloneqq \mathrm{id}, \quad R_\k \coloneqq 2i \frac{\partial}{\partial \tau} + \frac{\k}{v}. 
\end{align}
This identity is called {\it Bol's identity}, a proof can for example be found in \cite{thebook}*{Lemma 5.3}. 

\subsection{Locally harmonic Maass forms} \label{sec:LHMF}

In \cite{BKK}, so-called locally harmonic Maass forms, were introduced (for negative weights). See also \cite{hoevel} for the case of weight $0$.

\begin{defn}[\protect{\cite{BKK}*{Section $2$}}] \label{defn:LHMF}
	A function $f\colon \H\to\C$ is called a \textit{locally harmonic Maass form of weight $\k$} with exceptional set $E_D$, if it obeys the following four conditions:
	\begin{enumerate}[leftmargin=*,label=(\arabic*)]
		\item For every $\g\in\SL_2(\Z)$ we have $f|_\k\g=f$.
		
		\item For all $\t\in\H\setminus E_D$, there exists a neighborhood of $\t$, in which $f$ is real-analytic and in which we have $\De_{\k}(f)=0$.
		
		\item For every $\t\in E_D$, we have that
		\[
			f(\t) = \tfrac12\lim_{\eps\to0^+} (f(\t+i\eps)+f(\t-i\eps)).
		\]
		
		\item The function $f$ exhibits at most polynomial growth towards $i\infty$.
	\end{enumerate}
\end{defn}

Lastly, we define the various notions of singularities appearing in this paper.
\begin{defn}
Let $f\colon\H\setminus E_D\to\C$.
\begin{enumerate}[leftmargin=*,label=(\arabic*)]
\item We say that $f$ has {\it jumping singularities} on $E_D$ if there exists $\tau \in E_D$ such that
\[
	\lim_{\eps\to0^+} (f(\t+i\eps)-f(\t-i\eps)) \in \C \setminus \{0\}.
\]
Note that this limit depends on the geodesic $S_Q$ on which $\t$ is located. 

\item We say that $f$ has {\it continuously removable singularities} on $E_D$ if, for all $\t\in E_D$,
\[
	\lim_{\eps\to0^+} (f(\t+i\eps)-f(\t-i\eps)) = 0.
\]

\item We say that $f$ has {\it differentially removable singularities} on $E_D$ if $f$ is differentiable on $\H\setminus E_D$ and $f'$ has continuously removable singularities on $E_D$.
\end{enumerate}
\end{defn}

\section{Proof of Theorem \ref{thm:OmegaMmain} and of Corollary \ref{cor:omegacor}} \label{sec:Omegasec}

\subsection{Knopp's claims on \texorpdfstring{$\psi_{k+1,D}$}{\textpsi(k+1)D}}

We now discuss the initial claims of Knopp on $\psi_{k+1,D}$. 

\begin{prop}\label{prop:psiprop}
	\ \begin{enumerate}[leftmargin=*,label=\rm{(\arabic*)}]
		\item The functions $\psi_{k+1,D}$ converge absolutely on $\H$ and uniformly towards $i\infty$.
		
		\item For $n\in\N$, we have
		\[
			\D^{2n-1}\left(\Log\left(\frac{\t-\a_Q^-}{\t-\a_Q^+}\right)Q(\t,1)^{n-1}\right) = -i(2\pi)^{2n-1}(n-1)!^2D^{n-\frac12}\frac{1}{Q(\t,1)^n}.
		\]
		
		\item The functions $\psi_{k+1,D}$ satisfy $\psi_{k+1,D}(\t+1)=\psi_{k+1,D}(\t)$ and \eqref{eq:psitransf}.
	\end{enumerate}
\end{prop}

\begin{proof}
\begin{enumerate}[leftmargin=*,label=(\arabic*), wide, labelwidth=0pt, labelindent=0pt]
\item Let $Q=[a,b,c]$ and suppose that $v>1$. Since $\a_Q^\pm\in\R$ are the zeros of $Q$, we have $Q(\t,1)=a(\t-\a_Q^+)(\t-\a_Q^-)$ and $v>1$ implies that $\vert\t-\a_Q^+\vert>1$. Using $\vt{a}\ge1$ gives
	\[
		\vt{\Log\left(\frac{\t-\a_Q^-}{\t-\a_Q^+}\right)} \le \vt{\log\left(\frac{\vt{Q(\t,1)}}{\vt{a}\vt{\left(\t-\a_Q^+\right)}^2}\right)} + \pi \le \vt{\log\vt{Q(\t,1)}} + \pi,
	\]
	and (1) thus follows by the properties of $f_{\k,D}$ for $\k>1$ (see \cite{zagier75}).
\item We proceed by induction on $n$. The claims for $n=1$ and $n=2$ follow by computing
	\begin{equation} \label{eq:logdiff}
		\frac{\del}{\del\t} \Log\left(\frac{\t-\a_Q^-}{\t-\a_Q^+}\right) = -\frac{\sqrt{D}}{Q(\t,1)}, \quad
		\frac{\del^3}{\del\t^3} \left(\Log\left(\frac{\t-\a_Q^-}{\t-\a_Q^+}\right)Q(\tau,1)\right) = \frac{D^{\frac{3}{2}}}{Q(\tau,1)^2}, 
	\end{equation}
	utilizing Lemma \ref{lem:Qtricks} for $n=2$. To proceed with the induction step, we define for $n\in\N$
	\[
		\ff_n(\t) \coloneqq \Log\left(\frac{\t-\a_Q^-}{\t-\a_Q^+}\right)Q(\t,1)^{n-1},\qquad c_n \coloneqq (-1)^{n}(n-1)!^2.
	\]
	Since $Q$ is a polynomial of degree 2, we have, using the Leibniz rule,
	\begin{align*}
		\frac{\del^{2n+1}}{\del\t^{2n+1}} \mathfrak{f}_{n+1}(\t) &= \frac{\del^{2n+1}}{\del\t^{2n+1}} \left(\mathfrak{f}_{n}(\t)Q(\tau,1)\right) \\
		 &= \mathfrak{f}_n^{(2n+1)}(\t)Q(\t,1) + (2n+1)\mathfrak{f}_n^{(2n)}(\t)Q'(\t,1) + (2n+1)n\mathfrak{f}_n^{(2n-1)}Q''(\t,1).
	\end{align*}
	To apply the induction hypothesis, we write $\ff_n^{(2n)}(\t)=\frac{\del}{\del\t}\ff_n^{(2n-1)}(\t)$. Combining with the second identity of Lemma \ref{lem:Qtricks} then yields 
	\[
		\frac{\del^{2n+1}}{\del\t^{2n+1}} \mathfrak{f}_{n+1}(\t)= -\frac{n^2c_nD^{n+\frac12}}{Q(\t,1)^{n+1}}.
	\]
	Simplifying gives the claim.
\item Translation invariance of $\psi_{k+1,D}$ follows immediately from \eqref{eq:GammaOnQ} and the fact that
	\[
		[a,b,c]\circ\mat{1&1\\0&1}^{-1} = [a,-2a+b,a-b+c].
	\]
	Again using \eqref{eq:GammaOnQ} and the fact that
	\[
		[a,b,c]\circ\mat{0&-1\\1&0}^{-1} = [c,-b,a],
	\]
	we obtain that 
	\begin{align*}
		\t^{-2k-2}\psi_{k+1,D}\left(-\frac1\t\right)-\psi_{k+1,D}(\t) &= \sum_{Q\in\Qc_D} \frac{\Log\left(\frac{-\frac1\t-\frac{b-\sqrt{D}}{2c}}{-\frac1\t- \frac{b+\sqrt{D}}{2c}}\right) - \Log\left(\frac{\t-\frac{-b-\sqrt{D}}{2a}}{\t- \frac{-b+\sqrt{D}}{2a}}\right)}{Q(\t,1)^{k+1}}.
	\end{align*}
	Next, we recall that for $z,w\in\C\setminus\R$ 
	\begin{equation}\label{eq:logrule}
		\Log(z) - \Log(w) = \Log\left(\frac zw\right) + i\left(\Arg(z)-\Arg(w)-\Arg\left(\frac zw\right)\right).
	\end{equation}
	Choosing $z = \frac{-\frac1\t-\frac{b-\sqrt{D}}{2c}}{-\frac1\t- \frac{b+\sqrt{D}}{2c}}$, $w=\frac{\t-\frac{-b-\sqrt{D}}{2a}}{\t- \frac{-b+\sqrt{D}}{2a}}$ yields
	\begin{equation} \label{eq:modularinversion}
		\frac{z}{w} = \frac{\left(-\frac1\t-\frac{b-\sqrt{D}}{2c}\right) \left(\t-\frac{-b+\sqrt{D}}{2a}\right)}{\left(-\frac1\t-\frac{b+\sqrt{D}}{2c}\right) \left(\t-\frac{-b-\sqrt{D}}{2a}\right)} = \frac{\a_Q^+}{\a_Q^-} = \sgn(ac)\vt{\frac{\a_{Q}^+}{\a_{Q}^-}}.
	\end{equation}
	Hence $\Arg(z)=\Arg(\sgn(ac)w)$ and thus $\Arg(z)-\Arg(w)-\Arg(\frac zw)$ vanishes if $\sgn(ac)=1$. Thus the corresponding terms do not contribute to $\Arg(z)-\Arg(w)-\Arg(\frac zw)$. If $\sgn(ac)=-1$, then we extend $\Log$ by its principal value $\Log(x)=\log\vt{x}+\pi i$ for $x\in\R^-$. Then we use that
	\begin{align} \label{eq:argdiff}
		\Arg(-w) - \Arg(w) = -\sgn(\im(w))\pi,
	\end{align}
	and $\Arg(\frac zw) = \pi$. Hence, $\Arg(z)-\Arg(w)-\Arg(\frac zw)$ vanishes if $\sgn(ac)=-1$ and $\im(w) < 0$. To determine the sign of $\im(w)$, we calculate that
	\begin{align}\nonumber
		\frac{\t-\a_Q^-}{\t-\a_Q^+} &= \frac{\a_Q^+\a_Q^--\left(\a_Q^++\a_Q^-\right)u+u^2+v^2}{\vt{\t-\a_Q^+}^2} - \frac{i\left(\a_Q^+-\a_Q^-\right)v}{\vt{\tau-\a_Q^+}^2}\\
		\label{eq:loginput}	
		&= \frac{1}{\vt{\t-\a_Q^+}^2}\left(\frac{v Q_\tau}{a}-i\frac{\sqrt{D}}{a}v\right).
	\end{align}
	Thus, we have $\im(w)>0$ if and only if $a<0$. We conclude by \eqref{eq:logrule} and \eqref{eq:argdiff} that 
	\begin{align*}
	\Arg(z)-\Arg(w)-\Arg\left(\frac zw\right) =\begin{cases}
			-2\pi &\text{if } a<0<c,\\
			0 &\text{otherwise}. 
		\end{cases}
	\end{align*}
	Thus
	\[
		\t^{-2k-2}\psi_{k+1,D}\left(-\frac1\t\right)-\psi_{k+1,D}(\t) = \sum_{Q \in \Qc_D} \frac{\Log\left(\frac{\a_Q^+}{\a_Q^-}\right)}{Q(\tau,1)^{k+1}} - 2\pi i \sum_{\substack{Q\in\Qc_D\\a<0<c}} \frac{1}{Q(\t,1)^{k+1}}.
	\]
	By mapping $Q\mapsto-Q$, we arrive at 
	\begin{align} \label{eq:reallog}
\sum_{Q \in \Qc_D} \frac{\Log\left(\frac{\a_Q^+}{\a_Q^-}\right)}{Q(\tau,1)^{k+1}} = \sum_{Q \in \Qc_D} \frac{\log\vt{\frac{\a_Q^+}{\a_Q^-}}}{Q(\tau,1)^{k+1}} + \pi i \sum_{\substack{Q=[a,b,c] \in \Qc_D \\ \sgn(ac)=-1}} \frac{1}{Q(\tau,1)^{k+1}} = \sum_{Q \in \Qc_D} \frac{\log\vt{\frac{\a_Q^+}{\a_Q^-}}}{Q(\tau,1)^{k+1}}.
	\end{align}
This gives the claim. \qedhere
\end{enumerate}
\end{proof}

\begin{rmk}
	By \eqref{eq:loginput} the branch cut of $\Log(\frac{w-\a_Q^-}{w-\a_Q^+})$ is the interval $[\a_Q^-,\a_Q^+]$ or $[\a_Q^+,\a_Q^-]$. 
\end{rmk}

\subsection{The functions \texorpdfstring{$\rho_{k+1,D}$}{\textrho(k+1)D}, \texorpdfstring{$\varphi_{k+1,D}$}{\textphi(k+1)D}, and \texorpdfstring{$\lambda_{k+1,D}$}{\textlambda(k+1)D}}

Adapting the proof of Proposition \ref{prop:psiprop} (1), (3) we deduce the following results.

\begin{prop} \label{prop:rhoprop}
	\begin{enumerate}[leftmargin=*,label=\rm{(\arabic*)}]
		\item The functions $\rho_{k+1,D}$ converge absolutely on $\H\times\H^-$ and uniformly as $\t\to i\infty$ resp.\ $w\to-i\infty$.
		
		\item We have
		\begin{align*}
			\lim_{w \to -i\infty} \rho_{k+1,D}(\tau,w) = 0,  \qquad \lim_{\tau \to i\infty} \rho_{k+1,D}(\tau,w) = 0.
		\end{align*}
		
		\item 	Let $\tau \in \H$ and $w \in \H^-$. The functions $\rho_{k+1,D}$ satisfy
		\begin{align*}
			\rho_{k+1,D}(\tau+1,w+1) &= \rho_{k+1,D}(\tau,w), \\
			\tau^{-2k-2} \rho_{k+1,D}\left(-\frac{1}{\tau}, -\frac{1}{w}\right) - \rho_{k+1,D}(\tau,w) &= \sum_{Q \in \Qc_D} \frac{\log\vt{\frac{\alpha_Q^+}{\alpha_Q^-}}}{Q(\tau,1)^{k+1}} +  2\pi i \sum_{\substack{Q = [a,b,c]  \in \Qc_D \\ a < 0 < c}} \frac{1}{Q(\tau,1)^{k+1}}.
		\end{align*}
	\end{enumerate}
\end{prop}

\begin{proof}
Part (1) follows along the same lines as Proposition \ref{prop:psiprop} (1), and part (2) is an immediate consequence of uniform continuity of $\rho_{k+1,D}$ towards $i\infty$ (together with cuspidality of the $f_{\k,D}$ \cite{zagier75}). The first assertion of part (3) follows verbatim to translation invariance of $\psi_{k+1,D}$ (see Proposition \ref{prop:psiprop} (3)), and it remains to prove the second assertion of part (3). The idea is again an application of the rule from \eqref{eq:logrule} 
\begin{align} \label{eq:logrule2}
\Log(z_1) - \Log(z_2) = \Log\left(\frac{z_1}{z_2}\right) + i\left(\Arg(z_1)-\Arg(z_2)-\Arg\left(\frac{z_1}{z_2}\right)\right),
\end{align}
where $z_1 \coloneqq \frac{-\frac{1}{w}-\frac{b-\sqrt{D}}{2c}}{-\frac{1}{w}- \frac{b+\sqrt{D}}{2c}}$, and $z_2 \coloneqq \frac{w-\frac{-b-\sqrt{D}}{2a}}{w- \frac{-b+\sqrt{D}}{2a}}$. By \eqref{eq:modularinversion}, with $\tau$ replaced by $w$, we have
\begin{align*}
\frac{z_1}{z_2} = \frac{\left(-\frac{1}{w}-\frac{b-\sqrt{D}}{2c}\right) \left(w-\frac{-b+\sqrt{D}}{2a}\right)}{\left(-\frac{1}{w}-\frac{b+\sqrt{D}}{2c}\right) \left(w-\frac{-b-\sqrt{D}}{2a}\right)} = \frac{\a_Q^+}{\a_Q^-} = \sgn(ac)\vt{\frac{\a_Q^+}{\a_Q^-}}.
\end{align*}
It follows that $\Arg(z_1) = \Arg(\sgn(ac)z_2)$, and therefore $\Arg(z_1)-\Arg(z_2)-\Arg(\frac{z_1}{z_2}) = 0$ if $\sgn(ac) = 1$. If $\sgn(ac)=-1$, then \eqref{eq:argdiff} yields that $\Arg(-z_2)-\Arg(z_2) - \Arg(\frac{-z_1}{z_1}) = 0$ if $\im(z_2) < 0$ in addition. The sign of $\im(z_2)$ is given by \eqref{eq:loginput}, namely
	\[
		z_2 = \frac{w-\frac{-b-\sqrt{D}}{2a}}{w-\frac{-b+\sqrt{D}}{2a}}=\frac{\vt{w}^2 + \frac ba\re(w) + \frac ca - i\frac{\sqrt{D}}{a}\im(w)}{\vt{w+\frac{b+\sqrt{D}}{2a}}^2} \eqqcolon \frac{z_3}{\vt{w+\frac{b+\sqrt{D}}{2a}}^2}.
	\]
Hence, we have $\Arg(z_1)-\Arg(z_2)-\Arg\big(\frac{z_1}{z_2}\big) \neq 0$ if and only if $\sgn(ac) = -1$ and $\sgn(\im(z_3)) > 0$. Since $w \in \H^-$, we have
	\[
	\sgn(\im(z_3)) = -\sgn(a)\sgn(\im(w)) = \sgn(a),
	\]
	and thus we infer that 
	\begin{equation*}
\left(\Arg(z_1)-\Arg(z_2)-\Arg\left(\frac{z_1}{z_2}\right)\right) = \begin{cases}
-2\pi & \text{if } c < 0 < a, \\
0 & \text{otherwise}.
\end{cases}
	\end{equation*}
By changing $Q \mapsto-Q$ in both resulting expressions and using \eqref{eq:reallog}, we arrive at the claim.
\end{proof}

We next cite Parson's \cite{parson} result on her modular integrals $\varphi_{k+1,D}$.

\begin{lemma}[\protect{\cite{parson}*{Theorem 3.1}}] \label{lem:phitransf}
	The functions $\varphi_{k+1,D}$ satisfy 
\begin{align*}
\varphi_{k+1,D}(\t+1)&=\varphi_{k+1,D}(\t), \\
		\tau^{-2k-2}\varphi_{k+1,D}\left(-\frac{1}{\tau}\right) - \varphi_{k+1,D}(\tau) &= - \sum_{\substack{Q = [a,b,c]  \in \Qc_D \\ ac < 0}} \frac{\sgn{(Q)}}{Q(\tau,1)^{k+1}}
= 2\sum_{\substack{Q = [a,b,c]  \in \Qc_D \\ a < 0 < c}} \frac{1}{Q(\tau,1)^{k+1}}.
\end{align*}
	Furthermore, we have 
\begin{align*}
\lim_{\t\to i\infty}\varphi_{k+1,D}(\t)=0.
\end{align*}
\end{lemma}

We continue with some properties of $\l_{k+1,D}$.
\begin{prop}\label{prop:lambdaprop}
\ 
	\ \begin{enumerate}[leftmargin=*,label=\rm{(\arabic*)}]
		\item The functions $\l_{k+1,D}$ converge absolutely on $\H \times \H^-$, and uniformly as $\tau\to i\infty$ resp.\ $w\to -i\infty$.
		\item Let $\tau \in \H$ and $w \in \H^-$. The functions $\lambda_{k+1,D}$ are bimodular of weight $(2k+2, 0)$, that is 
		\begin{align*}
		\lambda_{k+1,D}\left(\tau+1,w+1\right) = \lambda_{k+1,D}\left(\tau,w\right), \qquad \lambda_{k+1,D}\left(-\frac{1}{\tau},-\frac{1}{w}\right) = \tau^{2k+2}\lambda_{k+1,D}\left(\tau,w\right).
		\end{align*}
		\item We have
		\begin{align*}
		\lim_{w \to -i\infty} \lambda_{k+1,D}(\tau,w) = -2\pi i\varphi_{k+1,D}(\tau), \qquad \lim_{\tau \to i\infty} \lambda_{k+1,D}(\tau,w) = 0.
		\end{align*}
		\item We have 
		\[
			\l_{k+1,D}(\t,w) = \sum_{Q\in\Qc_D} \frac{\Log\left(\frac{1+i\frac{Q_w}{\sqrt{D}}}{1-i\frac{Q_w}{\sqrt{D}}}\right)}{Q(\t,1)^{k+1}} = \sum_{Q\in\Qc_D} \frac{\Log\left(-\frac{\frac{Q_w}{\sqrt{D}}-i}{\frac{Q_w}{\sqrt{D}}+i}\right)}{Q(\t,1)^{k+1}}.
		\]
	\end{enumerate}
\end{prop}

\begin{proof}
\begin{enumerate}[leftmargin=*,label=(\arabic*), wide, labelwidth=0pt, labelindent=0pt]
\item By the definition of $\l_{k+1,D}$ in \eqref{eq:rhodef}, we have 
	\begin{align*}
		\vt{\lambda_{k+1,D}(\tau)} \leq 2\sum_{Q \in \Qc_D} \frac{\vt{\arctan\left(\frac{Q_{w}}{\sqrt{D}}\right)}}{\vt{Q(\tau,1)}^{k+1}} \leq \pi \sum_{Q \in \Qc_D} \frac{1}{\vt{Q(\tau,1)}^{k+1}}.
	\end{align*}
	The claim follows by the absolute convergence of the $f_{\k,D}$ on $\H$.
\item Bimodularity is a direct consequence of Lemma \ref{lem:InvarFactors} and \eqref{eq:GammaOnQ}.
\item The assumption that $D$ is not a square guarantees that the sum defining $\lambda_{k+1,D}$ runs over quadratic forms $Q=[a,b,c]$ with $ac\ne0$. To prove the first assertion, we observe that
	\begin{align} \label{eq:Qwasymp}
		\frac{Q_{w}}{\sqrt{D}} \asymp a\im(w)
	\end{align}
	as $\im(w) \to -\infty$, and hence
	\[
		\lim_{w\to -i\infty} \arctan\left(\frac{Q_{w}}{\sqrt{D}}\right) = -\frac\pi2\sgn(Q).
	\]
	The first claim follows by the definition of $\varphi_{k+1,D}$ in \eqref{eq:phidef}. As $a\ne0$, we have $\frac{1}{\vt{Q(\t,1)}^{k+1}}\to0$ for $\t\to i\infty$. The second claim follows by (1).
\item The claim follows by rewriting the arctangent in \eqref{eq:rhodef} in terms of logarithms. \qedhere
\end{enumerate}
\end{proof}

\subsection{Proof of Theorem \ref{thm:OmegaMmain} and of Corollary \ref{cor:omegacor}}

We conclude this section with the proofs of Theorem \ref{thm:OmegaMmain} and Corollary \ref{cor:omegacor}.

\begin{proof}[Proof of Theorem \ref{thm:OmegaMmain}]
\begin{enumerate}[leftmargin=*,label=(\arabic*), wide, labelwidth=0pt, labelindent=0pt]
\item This follows by combining Propositions \ref{prop:psiprop} (3), \ref{prop:rhoprop} (3), and \ref{prop:lambdaprop} (2) with Lemma \ref{lem:phitransf}.
\item This follows by combining \eqref{eq:OmegaMdef} with Propositions \ref{prop:rhoprop} (2) and \ref{prop:lambdaprop} (2).
\item Proposition \ref{prop:psiprop} (1) along with 
	\begin{align} \label{eq:loglimit}
		\lim_{\t\to i\infty} \Log\left(\frac{\tau-\alpha_{Q}^-}{\tau-\alpha_{Q}^+}\right) = 0
	\end{align}
	implies that $\psi_{k+1,D}$ is cuspidal. By Propositions \ref{prop:rhoprop} (2), \ref{prop:lambdaprop} (3), and Lemma \ref{lem:phitransf}, every function defining $\Omega_{k+1,D}$ in \eqref{eq:OmegaMdef} is cuspidal (with respect to $\tau$).
\item As each function defining $\Omega_{k+1,D}$ in \eqref{eq:OmegaMdef} is holomorphic as a function of $\tau$, we obtain the assertion with respect to $\tau$ directly. To verify that $\Omega_{k+1,D}$ is anti-holomorphic as a function of $w$, we compute by Lemmas \ref{lem:bkkidentity} and \ref{lem:holomzoo} (1) that
	\begin{align*}
		2i \frac{\del}{\del w} \arctan\left(\frac{Q_w}{\sqrt{D}}\right) = -\frac{\sqrt{D}}{Q(w,1)}.
	\end{align*}
	By \eqref{eq:logdiff}, we deduce that
	\begin{align*}
		2i \frac{\del}{\del w} \arctan\left(\frac{Q_w}{\sqrt{D}}\right) = \frac{\del}{\del w} \Log\left(\frac{w-\a_Q^-}{w-\a_Q^+}\right).
	\end{align*}
	By \eqref{eq:rhodef} and \eqref{eq:OmegaMdef}, we conclude that
	\begin{align*}
		\frac{\partial}{\partial w}\Omega_{k+1,D}(\tau,w) = 0.
	\end{align*}
\item We first inspect the functions $\psi_{k+1,D} - \rho_{k+1,D}$. By \eqref{eq:psidef} and \eqref{eq:rhodef} we have  
	\begin{align*}
		\psi_{k+1,D}(\tau) - \rho_{k+1,D}(\tau,\ol{\tau}) = \sum_{Q \in \Qc_D} \frac{\Log\left(\frac{\tau-\alpha_Q^-}{\tau-\alpha_Q^+}\right) - \Log\left(\frac{\overline{\tau}-\alpha_Q^-}{\overline{\tau}-\alpha_Q^+}\right)}{Q(\tau,1)^{k+1}}.
	\end{align*}
	We note that
	\begin{align*}
		\Log\left(\frac{\tau-\alpha_Q^-}{\tau-\alpha_Q^+}\right) - \Log\left(\frac{\overline{\tau}-\alpha_Q^-}{\overline{\tau}-\alpha_Q^+}\right) \equiv  \Log\left(\frac{\left(\tau-\alpha_{Q}^-\right)\left(\overline{\tau}-\alpha_{Q}^+\right)}{\left(\tau-\alpha_{Q}^+\right)\left(\overline{\tau}-\alpha_{Q}^-\right)}\right) \pmod{2\pi i},
	\end{align*}
	and we determine the multiple of $2\pi i$ now. From \eqref{eq:loginput}, we deduce that
	\begin{align} \label{eq:cayleytransform}
		\frac{\left(\t-\a_Q^-\right)\left(\ol\t-\a_Q^+\right)}{\left(\t-\a_Q^+\right) \left(\ol\t-\a_Q^-\right)} = \frac{\frac{Q_\t}{\sqrt{D}}-i}{\frac{Q_\t}{\sqrt{D}}+i}.
	\end{align}
	We use \eqref{eq:logrule} and hence need to compute 
	\begin{multline}\label{logdiff}
	\Log\left(\frac{\t-\a_Q^-}{\t-\a_Q^+}\right) - \Log\left(\frac{\ol\t-\a_Q^-}{\ol\t-\a_Q^+}\right) - \Log\left(\frac{\left(\t-\a_Q^-\right)\left(\ol\t-\a_Q^+\right)}{\left(\t-\a_Q^+\right)\left(\ol\t-\a_Q^-\right)}\right)\\
	= i\left(\Arg\left(\frac{\t-\a_Q^-}{\t-\a_Q^+}\right) - \Arg\left(\frac{\ol\t-\a_Q^-}{\ol\t-\a_Q^+}\right) - \Arg\left(\frac{\left(\t-\a_Q^-\right)\left(\ol\t-\a_Q^+\right)}{\left(\t-\a_Q^+\right)\left(\ol\t-\a_Q^-\right)}\right) \right).
	\end{multline}
	Note that for $z\in\C\setminus\R$ 
	\begin{equation*}
		\Arg(z) - \Arg\left(\ol z\right) - \Arg\left(\frac{z}{\ol z}\right) = \pi\left(1-\sgn\left(\re(z)\right)\right)\sgn\left(\im(z)\right).
	\end{equation*}
	We use this for $z=\frac{\t-\a_Q^-}{\t-\a_Q^+}$. By \eqref{eq:loginput}, \eqref{logdiff} thus becomes 
$\pi i(\sgn(Q_\t)-\sgn(Q))$. By \eqref{eq:cayleytransform},
	\[
		\psi_{k+1,D}(\tau) - \rho_{k+1,D}(\tau,\ol{\tau}) = \sum_{Q \in \Qc_D} \frac{\Log\left(\frac{\frac{Q_\t}{\sqrt{D}}-i}{\frac{Q_\t}{\sqrt{D}}+i}\right)} {Q(\tau,1)^{k+1}} + \pi i\sum_{\Qc\in\Qc_D} \frac{\sgn\left(Q_{\tau}\right)-\sgn(Q)}{Q(\t,1)^{k+1}}.
	\]
	Combining with \eqref{eq:phidef} gives
	\begin{align} \label{eq:Lambdatemp}
	\psi_{k+1,D}(\tau) - \rho_{k+1,D}(\tau,\ol{\tau}) + 2\pi i\varphi_{k+1,D}(\tau)
= \!\sum_{Q \in \Qc_D}\! \frac{\Log\left(\frac{\frac{Q_\t}{\sqrt{D}}-i}{\frac{Q_\t}{\sqrt{D}}+i}\right)} {Q(\tau,1)^{k+1}} + \pi i\sum_{\Qc\in\Qc_D} \frac{\sgn\left(Q_{\tau}\right)}{Q(\t,1)^{k+1}},
	\end{align}
	which is modular of weight $2k+2$ by \eqref{eq:GammaOnQ} and Lemma \ref{lem:InvarFactors}. To finish the proof, we inspect $\lambda_{k+1,D}(\tau,\ol{\tau})$. Combining $Q_{\ol\tau} = -Q_{\tau}$ with Proposition \ref{prop:lambdaprop} (4) yields
	\begin{align*}
		\lambda_{k+1,D}(\tau,\ol\tau) = -\sum_{Q\in\Qc_D} \frac{\Log\left(-\frac{\frac{Q_\tau}{\sqrt{D}}-i}{\frac{Q_\tau}{\sqrt{D}}+i}\right)}{Q(\t,1)^{k+1}}.
	\end{align*}
	By \eqref{eq:argdiff}, we obtain
	\begin{align*}
		\Log\left(\frac{\frac{Q_\t}{\sqrt{D}}-i}{\frac{Q_\t}{\sqrt{D}}+i}\right) - \Log\left(-\frac{\frac{Q_\tau}{\sqrt{D}}-i}{\frac{Q_\tau}{\sqrt{D}}+i}\right) = -\pi i \sgn(Q_{\tau}),
	\end{align*}
	from which we conclude that
	\begin{align} \label{eq:VanishingOnDiagonal}
		\Log\left(\frac{\t-\a_Q^-}{\tau-\a_Q^+}\right) - \Log\left(\frac{\ol{\tau}-\a_Q^-}{\ol{\tau}-\a_Q^+}\right) + \pi i\sgn(Q) + 2i\arctan\left(\frac{Q_{\ol\tau}}{\sqrt{D}}\right) = 0,
	\end{align}
	as well as the claim using \eqref{eq:OmegaMdef}. \qedhere
\end{enumerate}
\end{proof}

The proof of Corollary \ref{cor:omegacor} is along the same lines as the proof of Theorem \ref{thm:OmegaMmain}.

\begin{proof}[Proof of Corollary \ref{cor:omegacor}]
	Let $z \in \H$. Rearranging \eqref{eq:VanishingOnDiagonal} and substituting $\tau \mapsto z$ yields
	\begin{align*}
		- \Log\left(\frac{\overline{z}-\a_Q^-}{\overline{z}-\a_Q^+}\right) + \pi i\sgn(Q) + 2i\arctan\left(\frac{Q_{\overline{z}}}{\sqrt{D}}\right) = - \Log\left(\frac{z-\a_Q^-}{z-\a_Q^+}\right),
	\end{align*}
	which proves the first assertion.

Combining Propositions \ref{prop:psiprop} (3) and \ref{prop:rhoprop} (3) shows that
	\begin{multline*}
		\tau^{-2k-2}\rho_{k+1,D}\left(-\frac{1}{\tau},-\frac{1}{\mathfrak{z}}\right) - \rho_{k+1,D}(\tau,\mathfrak{z}) \\
		= \sum_{Q\in\Qc_D} \frac{\log\vt{\frac{\a_Q^+}{\a_Q^-}}}{Q(\t,1)^{k+1}} - 2\pi i \sgn\left(\im(\mathfrak{z})\right)\sum_{\substack{Q=[a,b,c]\in\Qc_D\\a<0<c}} \frac{1}{Q(\t,1)^{k+1}}, \qquad \mathfrak{z} \in \C \setminus \R.
	\end{multline*}
	In other words, $\psi_{k+1,D}$ and $\rho_{k+1,D}$ transform exactly the same on $\H \times \H$, from which weight $(2k+2,0)$ bimodularity of $\omega_{k+1,D}$ on $\H \times \H$ immediately follows. The other assertions are immediate consequences of the definitions.
\end{proof}

\section{The function \texorpdfstring{$\Lambda_{k+1,D}$}{\textLambda(k+1)D}} \label{sec:Lambdasec}

\subsection{Local cusp forms}
Recall the definition of $\La_{k+1,D}$ in \eqref{eq:Lambdadef}.

\begin{rmk}
	Let $d(z,w)$ denote the hyperbolic distance between $z,w\in\C$ with $\im(z)\im(w)>0$. Since $D>0$, we have (with $\t_{[a,b,c]}\coloneqq-\frac{b}{2a}+\frac{i}{2\vt{a}}\sqrt{D}$) $\frac{Q_\t}{\sqrt{D}}=\cosh(d(\t,\t_Q))$. This yields an alternative representation of $\La_{k+1,D}$ as well as of $\l_{k+1,D}$.
\end{rmk}

We next prove our claim for $\La_{k+1,D}$ from the introduction.
\begin{prop} \label{prop:Lambdacusp}
The functions $\Lambda_{k+1,D}$ are local cusp forms.
\end{prop}

\begin{proof}
	We observe that the $\La_{k+1,D}$ converge absolutely on $\H$ utilizing absolute convergence of the $f_{\k,D}$. We directly deduce that the $\La_{k+1,D}$ are holomorphic. Using Lemma \ref{lem:InvarFactors} and \eqref{eq:GammaOnQ} shows that the $\La_{k+1,D}$ are modular of weight $2k+2$. If $v>\frac{\sqrt{D}}{2}$, then $\sgn(Q_\t)=\sgn(Q)$ by \eqref{eq:Qwasymp}. Thus, cuspidality of the $\La_{k+1,D}$ follows by cuspidality of the $\varphi_{k+1,D}$ (see Lemma \ref{lem:phitransf}). The local behaviour and the jumping singularities are dictated by $\sgn(Q_\t)$.
\end{proof}

\subsection{The local behaviour of \texorpdfstring{$\Lambda_{k+1,D}$}{\textLambda(k+1)D}}

We next provide the behaviour of $\La_{k+1,D}$ on $E_D$.

\begin{prop} \label{prop:Lambdalocal}
	If $\t\in E_D$, then we have that
	\[
		\lim_{\eps\to0^+} (\La_{k+1,D}(\t+i\eps)-\La_{k+1,D}(\t-i\eps)) = 2\sum_{\substack{Q\in\Qc_D\\Q_\t=0}} \frac{\sgn(Q)}{Q(\t,1)^{k+1}}.
	\]
\end{prop}

\begin{rmk}
	The sum on the right-hand side is finite by \cite{BKK}*{Lemma 5.1 (1)}.
\end{rmk}

\begin{proof}[Proof of Proposition \ref{prop:Lambdalocal}]
	We adapt the proof of \cite{BKK}*{Proposition $5.2$}. We write
	\[
		\La_{k+1,D}(\t\pm i\eps) = \left(\sum_{\substack{Q\in\Qc_D\\Q_\t=0}} + \sum_{\substack{Q\in\Qc_D\\Q_\t\ne0}}\right) \frac{\sgn(Q_{\t\pm i\eps})}{Q(\t\pm i\eps,1)^{k+1}}.
	\]
	The properties of $f_{k+1,D}$ imply that $\Lambda_{k+1,D}$ converges absolutely on $\H$ and uniformly towards $i\infty$, which permits us to interchange the sums with the limit, and argue termwise.

	If $Q_\t\ne0$, then $\t\pm i\eps$ are in the same connected component of $\H\setminus E_D$ for $\eps>0$ sufficiently small. Combining with \cite{BKK}*{(5.4)}, we deduce that for $\eps>0$ sufficiently small
	\[
		\sgn([a,b,c]_{\t+i\eps}) = \sgn([a,b,c]_{\t-i\eps}) = \de\sgn(a),
	\]
	where
	\[
		\delta \coloneqq \sgn\left(\vt{\t+i\eps+\frac{b}{2a}}-\frac{\sqrt{D}}{2\vt{a}}\right) = \sgn\left(\vt{\t-i\eps+\frac{b}{2a}}-\frac{\sqrt{D}}{2\vt{a}}\right) = \pm1.
	\]
	Thus
\begin{align*}
		\lim_{\eps\to0^+} \!\left(\frac{\sgn(Q_{\t+i\eps})}{Q(\t+i\eps,1)^{k+1}} - \frac{\sgn(Q_{\t-i\eps})}{Q(\t-i\eps,1)^{k+1}}\right)
= \delta \lim_{\eps\to0^+} \!\left(\frac{\sgn(Q)}{Q(\t+i\eps,1)^{k+1}}-\frac{\sgn(Q)}{Q(\t-i\eps,1)^{k+1}}\right) = 0.
\end{align*}
	If $Q_\t=0$, then $\t\pm i\eps$ are in different connected components of $\H\setminus E_D$ for all $\eps>0$. This is justified by \cite{BKK}*{(5.6)}, namely, for every $\eps>0$,
	\[
		\vt{\t-i\eps+\frac{b}{2a}} - \frac{\sqrt{D}}{2\vt{a}} < \vt{\t+\frac{b}{2a}} - \frac{\sqrt{D}}{2\vt{a}} = 0 < \vt{\t+i\eps+\frac{b}{2a}} - \frac{\sqrt{D}}{2\vt{a}}.
	\]
	Combining with \cite{BKK}*{(5.4)} implies that $\sgn(Q_{\t\pm i\eps})=\pm\sgn(Q)$, and consequently
	\[
		\lim_{\eps\to0^+} \left(\frac{\sgn(Q_{\t+i\eps})}{Q(\t+i\eps,1)^{k+1}} - \frac{\sgn(Q_{\t-i\eps})}{Q(\t-i\eps,1)^{k+1}}\right) = 2\frac{\sgn(Q)}{Q(\t,1)^{k+1}}. \qedhere
	\]
\end{proof}

We next inspect the sum appearing in Proposition \ref{prop:Lambdalocal}. 

\begin{lemma} \label{lem:localnonzero}
The sum
\begin{align*}
\sum_{\substack{Q \in \Qc_D \\ Q_{\tau} = 0}} \frac{\sgn(Q)}{Q\left(\tau,1\right)^{k+1}}
\end{align*}
does not vanish identically on $E_D$.
\end{lemma}

\begin{proof}
	Let $\t\in E_D$. Then we have $\t\in S_\Qf$ for some $\Qf\in \Qc_D$. On the one hand, the sum in the lemma has a pole of order $k+1>0$ at $\a_\Qf^\pm$, and hence both limits
	\begin{equation*}
	\lim_{\substack{\tau \to \a_\Qf^\pm \\ \tau \in S_{\Qf}}} \vt{ \sum_{\substack{Q \in \Qc_D \\ Q_{\tau} = 0}} \frac{\sgn(Q)}{Q\left(\tau,1\right)^{k+1}} }
	\end{equation*}
	tend towards $\infty$. On the other hand, the sum is continuous on $S_\Qf$, and the contribution from the terms corresponding to $Q \neq \Qf$ is finite at $\a_\Qf^\pm$.
\end{proof}

\subsection{The local behaviour of \texorpdfstring{$\Ec_{\La_{k+1,D}}$}{E\textLambda(k+1)D} and \texorpdfstring{$\Lambda_{k+1,D}^*$}{\textLambda(k+1)D*}}

We next prove that the Eichler integrals of $\La_{k+1,D}$ exist on $E_D$.

\begin{prop} \label{prop:Eichlerlocal}
	Let $\t\in E_D$. Then we have
	\begin{align*}
		\lim_{\eps\to0^+} \left(\Ec_{\La_{k+1,D}}(\t+i\eps)-\Ec_{\La_{k+1,D}}(\t-i\eps)\right) &= -\frac{2(2\pi i)^{2k+1}}{(2k)!}\int_0^{i\infty} \sum_{\substack{Q\in\Qc_D\\Q_{\t+w}=0}} \frac{\sgn(Q)}{Q(\t+w,1)^{k+1}}w^{2k} dw,\\
		\lim_{\eps\to0^+} \left(\La_{k+1,D}^*(\t+i\eps)-\La_{k+1,D}^*(\t-i\eps)\right) &= -\frac{2}{(2i)^{2k+1}} \int_{2iv}^{i\infty} \sum_{\substack{Q\in\Qc_D\\Q_{\t-w}=0}} \frac{\sgn(Q)}{Q(\t-w,1)^{k+1}}w^{2k} dw.
	\end{align*}
\end{prop}

\begin{rmk}
If $\t+ w\notin E_D$ resp.\ $\t - w\notin E_D$, then the sums inside the integrands on the right-hand sides of Propostion \ref{prop:Eichlerlocal} are empty. If $\t + w\in E_D$ resp.\ $\t - w \in E_D$, then the sums inside the integrands are finite as remarked after Proposition \ref{prop:Lambdalocal}. Thus, each integral runs over a bounded domain, because the integrands vanish as soon as\footnote{If $\im(\t+w)>\frac{\sqrt{D}}{2}$, then $\t+w$ lies in the unbounded component of $\H\setminus E_D$.} $\im(\t+w)>\frac{\sqrt{D}}{2}$ or $\t-w$ moves out of $\H$.  Hence, the integrals on the right-hand side of Propostion \ref{prop:Eichlerlocal} exist.
\end{rmk}

\begin{proof}[Proof of Proposition \ref{prop:Eichlerlocal}]
	As $\t\pm i\eps\notin E_D$ for $\eps>0$, we utilize \eqref{eq:Lambdadef}. Changing variables gives
	\begin{align*}
		&\lim_{\eps\to0^+} \left(\Ec_{\La_{k+1,D}}(\t+i\eps)-\Ec_{\La_{k+1,D}}(\t-i\eps)\right)\\
		&\quad\quad= -\frac{(2\pi i)^{2k+1}}{(2k)!}\lim_{\eps\to0^+} \int_0^{i\infty} \sum_{Q\in\Qc_D} \left(\frac{\sgn(Q_{\t+i\eps+w})}{Q(\t+i\eps+w,1)^{k+1}} - \frac{\sgn(Q_{\t-i\eps+w})}{Q(\t-i\eps+w,1)^{k+1}}\right)w^{2k} dw, \\
		&\lim_{\eps\to0^+} \left(\La_{k+1,D}^*(\t+i\eps)-\La_{k+1,D}^*(\t-i\eps)\right) = \frac{1}{(2i)^{2k+1}} \\ 
		& \quad\times \lim_{\eps\to0^+} \left(
\int_{2i(v-\eps)}^{i\infty} \sum_{Q\in\Qc_D} \frac{\sgn\left(Q_{\t-w-i\eps}\right) w^{2k}}{Q(\t-w-i\eps,1)^{k+1}} dw
-\int_{2i(v+\eps)}^{i\infty} \sum_{Q\in\Qc_D} \frac{\sgn\left(Q_{\t-w+i\eps}\right) w^{2k}}{Q(\t-w+i\eps,1)^{k+1}} dw 
\right),
	\end{align*}
	where we use $Q_{\ol z}=-Q_z$ for $\La_{k+1,D}^*$. We next justify interchanging the limit $\eps\to0^+$ with the holomorphic Eichler integral. By \eqref{eq:Lambdadef}, $\La_{k+1,D}$ vanishes at $i\infty$, and converges uniformly towards $i\infty$ as the sign-function is bounded (as $f_{\k,D}$ converges uniformly towards $i\infty$ for $\k>1$). By modularity of $\La_{k+1,D}$, both assertions hold towards $0$ as well. So the integral converges uniformly, and this permits the exchange of the limit $\eps\to0^+$ with the integral. 

	We consider the holomorphic Eichler integral first. If $\t+w\notin E_D$, then the limit inside the integral vanishes, because $\t+w+i\eps$ and $\t+w-i\eps$ are in the same connected component for $\eps$ sufficiently small. If $\t+w\in E_D$, then we apply Proposition \ref{prop:Lambdalocal} to obtain
	\begin{align*}
		\lim_{\eps\to0^+} &\left(\mathcal{E}_{\Lambda_{k+1,D}}(\t+i\eps)-\mathcal{E}_{\Lambda_{k+1,D}}(\t-i\eps)\right) \\
		&\quad\quad= -\frac{(2\pi i)^{2k+1}}{(2k)!} \int_0^{i\infty} \lim_{\eps\to0^+} \left(\Lambda_{k+1,D}(\tau+w+i\eps) - \Lambda_{k+1,D}(\tau+w-i\eps)\right) w^{2k} dw \\
		&\quad\quad= -\frac{2(2\pi i)^{2k+1}}{(2k)!} \int_0^{i\infty} \sum_{\substack{Q\in\Qc_D\\Q_{\t+w}=0}} \frac{\sgn(Q)}{Q(\t+w,1)^{k+1}}w^{2k} dw.
	\end{align*}

	Now, we treat the non-holomorphic Eichler integrals, and first split one of them as
\begin{multline*}
		\int_{2i(v-\eps)}^{i\infty} \sum_{Q \in \Qc_D} \frac{\sgn\left(Q_{\tau-w-i\eps}\right)}{Q(\t-w-i\eps,1)^{k+1}}w^{2k} dw \\
= \left(\int_{2i(v-\eps)}^{2i(v+\eps)} + \int_{2i(v+\eps)}^{i\infty}\right) \sum_{Q \in \Qc_D} \frac{\sgn\left(Q_{\tau-w-i\eps}\right)}{Q(\t-w-i\eps,1)^{k+1}}w^{2k} dw.
\end{multline*}
	We note that
	\begin{align*}
		\lim_{\eps\to0^+} \int_{2i(v-\eps)}^{2i(v+\eps)} \frac{\sgn\left(Q_{\tau-w-i\eps}\right)}{Q(\t-w-i\eps,1)^{k+1}}w^{2k} dw = 0,
	\end{align*}
	because the integrand is bounded in the domain of integration, which has measure $0$ as $\varepsilon \to 0^+$. Hence, it remains to consider the integral from $2i(v+\eps)$ to $i\infty$. If $\t-w\notin E_D$, then we have
	\begin{align*}
		\lim_{\eps\to0^+} \int_{2i(v+\eps)}^{i\infty} \left(-\frac{\sgn\left(Q_{\tau-w+i\eps}\right)}{Q(\t-w+i\eps,1)^{k+1}} + \frac{\sgn\left(Q_{\tau-w-i\eps}\right)}{Q(\t-w-i\eps,1)^{k+1}}\right)w^{2k} dw = 0,
	\end{align*}
	as in the previous case, because $\t-w\pm i\eps$ are in the same connected component for $\eps$ sufficiently small. If $\t-w\in E_D$, then we obtain
	\begin{multline*}
		\lim_{\eps\to0^+} \int_{2i(v+\eps)}^{i\infty} \sum_{Q \in \Qc_D} \left(-\frac{\sgn\left(Q_{\tau-w+i\eps}\right)}{Q(\t-w+i\eps,1)^{k+1}} + \frac{\sgn\left(Q_{\tau-w-i\eps}\right)}{Q(\t-w-i\eps,1)^{k+1}}\right)w^{2k} dw \\
		= -2 \int_{2iv}^{i\infty} \sum_{\substack{Q\in\Qc_D\\Q_{\t-w}=0}} \frac{\sgn(Q)}{Q(\t-w,1)^{k+1}} w^{2k} dw
	\end{multline*}
	by Proposition \ref{prop:Lambdalocal} exactly as in the previous case.
\end{proof}

\section{The function \texorpdfstring{$\Psi_{-k,D}$}{\textPsi(-k)D} and the proof of Theorem \ref{thm:Psimain}} \label{sec:Psisec}

\subsection{Convergence of \texorpdfstring{$\Psi_{-k,D}$}{\textPsi(-k)D}}

We first establish convergence of $\Psi_{-k,D}$.

\begin{prop}\label{prop:Psiconv}
	The sum defining $\Psi_{-k,D}$ converges compactly on $\H\setminus E_D$, and does not converge on $E_D$.
\end{prop}

\begin{proof}
	 If $\t\in\H\setminus E_D$, then $\sgn(Q_\t)=\pm1$ and thus the claim follows directly by \cite{BKK}*{Proposition 4.1} after summing over all narrow equivalence classes there. (The class number of positive discriminants is finite.) If $\t\in E_D$, then the incomplete $\b$-function reduces to a constant depending only on $k$ according to Lemma \ref{lem:bkkidentity}. Hence, the sum defining $\Psi_{-k,D}$ does not converge on $E_D$ as the sum is infinite and $\b(1;k+\frac12,\frac12)\ne0$.
\end{proof}

\subsection{Behaviour of \texorpdfstring{$\Psi_{-k,D}$}{\textPsi(-k)D} under differentiation}

We inspect the behaviour of $\Psi_{-k,D}$ under differential operators.

\begin{prop}\label{prop:Psidiff}
	Let $\t\in\H\setminus E_D$. 
	\begin{enumerate}[leftmargin=*,label=\rm{(\arabic*)}]
		\item We have
		\[
			\xi_{-2k}(\Psi_{-k,D}(\tau)) = D^{k+\frac12}\La_{k+1,D}(\tau).
		\]
		\item We have
		\[
			\D^{2k+1}(\Psi_{-k,D}(\tau)) = -\frac{D^{k+\frac12}(2k)!}{(4\pi)^{2k+1}}\La_{k+1,D}(\tau).
		\]
		\item We have
		\[
			\De_{-2k}(\Psi_{-k,D}(\tau)) = 0.
		\]
	\end{enumerate}
\end{prop}

Define
\[
	g_n^{[1]}(\t) \coloneqq Q(\t,1)^n\b\left(\frac{Dv^2}{\vt{Q(\t,1)}^2};n+\frac12,\frac12\right),\qquad n \in \N_0.
\]
The proof of Proposition \ref{prop:Psidiff} is based on the following three technical lemmas.

\begin{lemma}\label{lem:fnrec}
	We have for $n\in\N_0$
	\[
		g_{n+1}^{[1]}(\t) = \frac{n+\frac12}{n+1}Q(\t,1)g_n^{[1]}(\t) - \frac{D^{n+\frac12}}{n+1}\frac{v^{2n+2}\vt{Q_\t}}{Q\left(\ol\t,1\right)^{n+1}}.
	\]
\end{lemma}

\begin{proof}
	By \cite{nist}*{(8.17.20)}, we have that
	\[
		\frac{\b(x;a,b)}{\b(1;a,b)} = \frac{\b(x;a+1,b)}{\b(1;a+1,b)} + \frac{x^a(1-x)^b}{a\b(1;a,b)}.
	\]
	This gives that
	\begin{multline*}
		\b\left(\frac{Dv^2}{\vt{Q(\t,1)}^2};n+\frac32,\frac12\right)\\
		= \frac{\b\left(1;n+\frac32,\frac12\right)}{\b\left(1;n+\frac12,\frac12\right)} \left(\b\left(\frac{Dv^2}{\vt{Q(\t,1)}^2};n+\frac12,\frac12\right) - \frac{\left(\frac{Dv^2}{\vt{Q(\t,1)}^2}\right)^{n+\frac12} \left(1-\frac{Dv^2}{\vt{Q(\t,1)}^2}\right)^\frac12}{n+\frac12}\right).
	\end{multline*}
	Using Lemma \ref{lem:bkkidentity}, we compute
	\[
		\left(\frac{Dv^2}{\vt{Q(\t,1)}^2}\right)^{n+\frac12} \left(1-\frac{Dv^2}{\vt{Q(\t,1)}^2}\right)^\frac12 = \frac{D^{n+\frac12}v^{2n+2}\vt{Q_\t}}{\vt{Q(\t,1)}^{2n+2}}, 
	\]
	and since $\frac{\b(1;n+\frac32,\frac12)}{\b(1;n+\frac12,\frac12)}=\frac{n+\frac12}{n+1}$, we obtain the claim.
\end{proof}

Lemma \ref{lem:fnrec} motivates to define the auxiliary function
\[
	g_n^{[2]}(\t) \coloneqq \frac{D^{n-\frac12}v^{2n}\vt{Q_\t}}{Q\left(\ol\t,1\right)^n}.
\]
The second technical lemma treats the image of $g_{n+1}^{[2]}$ under differentiation.

\begin{lemma}\label{lem:gnzero}
	We have for $n\in\N$
	\[
		\frac{\del^{2n+1}}{\del\t^{2n+1}} g_n^{[2]}(\t) = 0.
	\]
\end{lemma}

\begin{proof}
	We prove the claim by induction. If $n=1$, then the claim follows by applying Lemma \ref{lem:holomzoo} (1) three times. For the induction step, Lemma \ref{lem:holomzoo} (1) yields that
	\[
		\frac{\del}{\del\t} \left(v^{\ell+2}Q_\t\right) = -\frac i2\ell v^{\ell+1}Q_\t + \frac i2v^\ell Q\left(\ol\t,1\right)
	\]
	for every $\ell\in\N_0$. Noting that $\frac{\del^{\ell+1}}{\del\t^{\ell+1}}(v^\ell Q(\ol\t,1))=0$, we obtain
	\[
		\frac{\del^{\ell+2}}{\del\t^{\ell+2}} \left(v^{\ell+1}Q_\t\right) = -\frac i2(\ell+1) \frac{\del^{\ell+1}}{\del\t^{\ell+1}} \left(v^\ell Q_\t\right).
	\]
	Consequently, we find that
	\[
		\frac{\del^{2n+3}}{\del\t^{2n+3}} g_{n+1}^{[2]}(\t) = -\frac{D(2n+2)(2n+1)}{4Q\left(\ol\t,1\right)} \frac{\del^{2n+1}}{\del\t^{2n+1}} g_n^{[2]}(\t).
	\]
	The right-hand side vanishes by the induction hypothesis, as desired.
\end{proof}

The third lemma contains the main technical claim.

\begin{lemma}\label{lem:bolPsiInduction}
	We have for $n\in\N_0$
	\[
		\frac{\del^{2n+1}}{\del\t^{2n+1}} g_n^{[1]}(\t) = \frac{i(-1)^{n+1}D^{n+\frac12}(2n)!\sgn(Q_\t)}{2^{2n}Q(\t,1)^{n+1}}.
	\]
\end{lemma}

\begin{proof}
	We prove the lemma by induction.\\
	{\bf Step $1$: The case $n=0$}\\
	We apply the Fundamental Theorem of Calculus, Lemma \ref{lem:bkkidentity}, and Lemma \ref{lem:holomzoo}, yielding
	\begin{equation}\label{eq:betaHdiff}
		\frac{\del}{\del\t} \b\left(\frac{Dv^2}{\vt{Q(\t,1)}^2};n+\frac12,\frac12\right) = -\frac{iD^{n+\frac12}v^{2n}\sgn(Q_\t)}{\vt{Q(\t,1)}^{2n}Q(\t,1)}
	\end{equation}
	for every $n\in\N_0$. In particular, this proves the desired identity for $n=0$.\\
	{\bf Step $2$: The case $n=1$}\\
	Using \eqref{eq:betaHdiff} and the first identity of Lemma \ref{lem:Qtricks}, we compute that
	\[
		R_{-2n}\left(g_n^{[1]}(\t)\right) = -2nQ_\t\frac{g_n^{[1]}(\t)}{Q(\t,1)} + \frac{2D^{n+\frac12}v^{2n}\sgn(Q_\t)}{Q\left(\ol\t,1\right)^nQ(\t,1)}.
	\]
	Lemma \ref{lem:fnrec} with $n\mapsto n-1$ gives
	\[
		\frac{g_n^{[1]}(\t)}{Q(\t,1)} = \frac{n-\frac12}{n}g_{n-1}^{[1]}(\t) - \frac{D^{n-\frac12}}{n}\frac{v^{2n}\sgn(Q_\t)Q_\t}{Q\left(\ol\t,1\right)^nQ(\t,1)}.
	\]
	Plugging into the previous equation and applying Lemma \ref{lem:bkkidentity} yields
	\[
		R_{-2n}\left(g_n^{[1]}(\t)\right) = -(2n-1)Q_\t g_{n-1}^{[1]}(\t) + \frac{2D^{n-\frac12}v^{2n-2}\sgn(Q_\t)}{Q\left(\ol\t,1\right)^{n-1}}.
	\]
	We compute
	\begin{align*}
		R_{2-2n}\left(\frac{2D^{n-\frac12}v^{2n-2}\sgn(Q_\t)}{Q\left(\ol\t,1\right)^{n-1}}\right) &= 0,\\
		R_{2-2n}\left(Q_\t g_{n-1}^{[1]}(\t)\right) &= Q_\t R_{2-2n}\left(g_{n-1}^{[1]}(\t)\right) - g_{n-1}^{[1]}(\t)\frac{Q\left(\ol\t,1\right)}{v^2}
	\end{align*}
	by Lemma \ref{lem:holomzoo} (1). We infer that
	\[
		R_{2-2n}\circ R_{-2n}\left(g_n^{[1]}(\t)\right) = -(2n-1)\left(Q_\t R_{2-2n}\left(g_{n-1}^{[1]}(\t)\right)-g_{n-1}^{[1]}(\t)\frac{Q\left(\ol\t,1\right)}{v^2}\right).
	\]
	
	Now, we suppose that $n=1$. Then the previous equation gives 
	\[
		R_0\circ R_{-2}\left(g_1^{[1]}(\t)\right) = -Q_\t R_0\left(g_0^{[1]}(\t)\right) + g_0^{[1]}(\t)\frac{Q\left(\ol\t,1\right)}{v^2}.
	\]
	We then compute, using \eqref{eq:betaHdiff}
	\[
		R_0\left(g_0^{[1]}(\t)\right) = 2i\frac{\del}{\del\t} \b\left(\frac{Dv^2}{\vt{Q(\t,1)}^2};\frac12,\frac12\right) = \frac{2D^\frac12\sgn(Q_\t)}{Q(\t,1)}.
	\]
	Combining this with the previous equation we obtain
	\[
		R_2\circ R_0\circ R_{-2}\left(g_1^{[1]}(\t)\right) = R_2\left(-2Q_\t\frac{D^\frac12\sgn(Q_\t)}{Q(\t,1)} + \b\left(\frac{Dv^2}{\vt{Q(\t,1)}^2};\frac12,\frac12\right)\frac{Q\left(\ol\t,1\right)}{v^2}\right).
	\]
	By Lemma \ref{lem:holomzoo} (1) and \eqref{eq:betaHdiff}, we calculate that
	\begin{multline*}
		\frac{\del}{\del\t} \left(-2Q_\t\frac{D^\frac12\sgn(Q_\t)}{Q(\t,1)} + \b\left(\frac{Dv^2}{\vt{Q(\t,1)}^2};\frac12,\frac12\right)\frac{Q\left(\ol\t,1\right)}{v^2}\right) = -\frac{iQ\left(\ol\t,1\right)}{v^2}\frac{D^\frac12\sgn(Q_\t)}{Q(\t,1)} \\ 
+ 2Q_\t\frac{D^\frac12\sgn(Q_\t)}{Q(\t,1)^2} Q'(\t,1) - \frac{iD^\frac12\sgn(Q_\t)}{Q(\t,1)}\frac{Q\left(\ol\t,1\right)}{v^2}
		+ \frac{i}{v^3}\b\left(\frac{Dv^2}{\vt{Q(\t,1)}^2};\frac12,\frac12\right)Q\left(\ol\t,1\right).
	\end{multline*}
	Using Lemma \ref{lem:bkkidentity} and the first identity of Lemma \ref{lem:Qtricks}, we obtain
\begin{align*}
		R_2\circ R_0\circ R_{-2}\left(g_1^{[1]}(\t)\right) &= 2i\left(-2i\frac{Q\left(\ol\t,1\right)}{v^2}\frac{\sqrt{D}\sgn\left(Q_\t\right)}{Q(\t,1)} + 2\frac{\sqrt{D}|Q_\t|Q'(\t,1)}{Q(\t,1)^2}\right.\\
		&\hspace{-3.2cm}\left.+ i\b\left(\frac{Dv^2}{|Q(\t,1)|^2};\frac12,\frac12\right)\frac{Q\left(\ol\t,1\right)}{v^3}\right) + \frac2v\left(-\frac{2\sqrt{D}|Q_\t|}{Q(\t,1)} + \b\left(\frac{Dv^2}{|Q(\t,1)|^2};\frac12,\frac12\right)\frac{Q\left(\ol\t,1\right)}{v^2}\right)\\
		&= 4\sqrt{D}\frac{|Q(\t,1)|^2}{v^2}\frac{\sgn(Q_\t)}{Q(\t,1)^2} - \frac{4\sqrt{D}|Q_\t|}{Q(\t,1)^2}\left(-iQ'(\t,1)+\frac{Q(\t,1)}{v}\right)\\
		&= 4\sqrt{D}\left(D+Q_\t^2\right)\frac{\sgn(Q_\t)}{Q(\t,1)^2} - \frac{4\sqrt{D}\sgn(Q_\t)}{Q(\t,1)^2}Q_\t^2 = \frac{4D^\frac32\sgn(Q_\t)}{Q(\t,1)^2}.
	\end{align*}
	We can then directly conclude the claim using Bol's identity \eqref{eq:bolidentity}.\\
	{\bf Step $3$: Application of Lemmas \ref{lem:fnrec} and \ref{lem:gnzero} and reducing to $2n+2$ derivatives}\\
	Employing Lemma \ref{lem:fnrec} and Lemma \ref{lem:gnzero} with $n\mapsto n+1$ yields
	\begin{equation}\label{eq:rewriteclaim}
		\frac{\del^{2n+3}}{\del\t^{2n+3}} g_{n+1}^{[1]}(\t) = \frac{n+\frac12}{n+1}\frac{\del^{2n+3}}{\del\t^{2n+3}} \left(Q(\t,1)g_n^{[1]}(\t)\right).
	\end{equation}
	By \eqref{eq:betaHdiff}, we compute that
	\[
		\frac{\del}{\del\t} \left(Q(\t,1)g_n^{[1]}(\t)\right) = (n+1)Q(\t,1)^nQ'(\t,1)\b\left(\frac{Dv^2}{\vt{Q(\t,1)}^2};n+\frac12,\frac12\right) - \frac{iD^{n+\frac12}v^{2n}\sgn(Q_\t)}{Q\left(\ol\t,1\right)^n}.
	\]
	We observe that the final term gets annihilated by differentiating $2n+1$ times and thus
	\begin{align*}
		\frac{n\!+\!\frac12}{n\!+\!1}\frac{\del^{2n+3}}{\del\t^{2n+3}} \!\left(Q(\t,1)g_n^{[1]}(\t)\right) 
		\!=\! \left(n\!+\!\frac12\right)\!\frac{\del^{2n+2}}{\del\t^{2n+2}} \!\left(\!Q(\t,1)^nQ'(\t,1)\b\!\left(\frac{Dv^2}{\vt{Q(\t,1)}^2};n\!+\!\frac12,\frac12\right)\!\right)\!.
	\end{align*}
	{\bf Step $4$: Reducing to $2n+1$ derivatives}\\
	By \eqref{eq:betaHdiff}, we furthermore calculate that
	\begin{multline*}
		\frac{\del}{\del\t} \left(Q(\t,1)^nQ'(\t,1)\b\left(\frac{Dv^2}{\vt{Q(\t,1)}^2};n+\frac12,\frac12\right)\right) = Q(\t,1)^nQ''(\t,1)\b\left(\frac{Dv^2}{\vt{Q(\t,1)}^2};n+\frac12,\frac12\right)  \\
		+ nQ(\t,1)^{n-1}Q'(\t,1)^2\b\left(\frac{Dv^2}{\vt{Q(\t,1)}^2};n+\frac12,\frac12\right)
- iQ(\t,1)^nQ'(\t,1)\frac{D^{n+\frac12}v^{2n}\sgn(Q_\t)}{\vt{Q(\t,1)}^{2n}Q(\t,1)}.
	\end{multline*}
	By the first identity of Lemma \ref{lem:Qtricks}, the final term may be rewritten as
	\[
		-iQ(\t,1)^nQ'(\t,1)\frac{D^{n+\frac12}v^{2n}\sgn(Q_\t)}{\vt{Q(\t,1)}^{2n}Q(\t,1)} = -\frac{D^{n+\frac12}v^{2n-1}\sgn(Q_\t)}{Q\left(\ol\t,1\right)^n} + \frac{D^{n+\frac12}v^{2n}\vt{Q_\t}}{Q\left(\ol\t,1\right)^nQ(\t,1)}.
	\]
	Again the final term gets annihilated upon differentiating $2n+1$ times. Consequently, we obtain, by the second identity of Lemma \ref{lem:Qtricks},
	\begin{multline*}
		\frac{n+\frac12}{n+1}\frac{\del^{2n+3}}{\del\t^{2n+3}} \left(Q(\t,1)g_n^{[1]}(\t)\right) = \left(n+\frac12\right)\frac{\del^{2n+1}}{\del\t^{2n+1}} \left(DnQ(\t,1)^{n-1}\b\left(\frac{Dv^2}{\vt{Q(\t,1)}^2};n+\frac12,\frac12\right)\right.\\
		\left.+ (2n+1)Q(\t,1)^nQ''(\t,1)\b\left(\frac{Dv^2}{\vt{Q(\t,1)}^2};n+\frac12,\frac12\right) + \frac{D^{n+\frac12}v^{2n}\vt{Q_\t}}{Q\left(\ol\t,1\right)^nQ(\t,1)}\right)\\
		= \left(n+\frac12\right)\frac{\del^{2n+1}}{\del\t^{2n+1}} \left(Dn\frac{g_n^{[1]}(\t)}{Q(\t,1)}+(2n+1)Q''(\t,1)g_n^{[1]}(\t)+\frac{Dg_n^{[2]}(\t)}{Q(\t,1)}\right).
	\end{multline*}
	{\bf Step $5$: Application of the induction hypothesis}\\
	We use Lemma \ref{lem:fnrec} with $n\mapsto n-1$, to obtain
	\[
		\frac{g_n^{[1]}(\t)}{Q(\t,1)} = \frac{n-\frac12}{n}g_{n-1}^{[1]}(\t) - \frac{g_n^{[2]}(\t)}{nQ(\t,1)},
	\]
	and hence, using step 4,
	\begin{multline*}
		\frac{n+\frac12}{n+1}\frac{\del^{2n+3}}{\del\t^{2n+3}} \left(Q(\t,1)g_n^{[1]}(\t)\right) \\
		= \left(n+\frac12\right)\frac{\del^{2n+1}}{\del\t^{2n+1}} \left(D\left(n-\frac12\right)g_{n-1}^{[1]}(\t)+(2n+1)Q''(\t,1)g_n^{[1]}(\t)\right).
	\end{multline*}
	The induction hypothesis for $n$ and $n-1$, and the fact that $Q''(\t,1)$ is independent of $\t$ gives
	\begin{multline*}
		\frac{n+\frac12}{n+1}\frac{\del^{2n+3}}{\del\t^{2n+3}} \left(Q(\t,1)g_n^{[1]}(\t)\right)\\
		= \frac{(-1)^niD^{n+\frac12}\left(n+\frac12\right)(2n)!\sgn(Q_\t)}{4^{n}} \left(\frac1n\frac{\del^2}{\del\t^2}\frac{1}{Q(\t,1)^n} - (2n+1)\frac{Q''(\t,1)}{Q(\t,1)^{n+1}}\right).
	\end{multline*}
	{\bf Step $6$: Simplifying the expressions}\\
	Using the second identity of Lemma \ref{lem:Qtricks}, we compute
	\[
		\frac1n\frac{\del^2}{\del\t^2} \frac{1}{Q(\t,1)^n} - (2n+1)\frac{Q''(\t,1)}{Q(\t,1)^{n+1}} = \frac{D(n+1)}{Q(\t,1)^{n+2}}.
	\]
	Inserting this into the result from step $5$ yields
	\[
		\frac{n+\frac12}{n+1}\frac{\del^{2n+3}}{\del\t^{2n+3}} \left(Q(\t,1)g_n^{[1]}(\t)\right) = \frac{(-1)^ni\left(n+\frac12\right)\left(n+1\right)(2n)!D^{n+\frac32}\sgn(Q_\t)}{4^{n}Q(\t,1)^{n+2}}.
	\]
	By \eqref{eq:rewriteclaim}, we ultimately arrive at the claim of the lemma (with $n\mapsto n+1$).
\end{proof}

We are now ready to prove Proposition \ref{prop:Psidiff}.
\begin{proof}[Proof of Proposition \ref{prop:Psidiff}]
\begin{enumerate}[leftmargin=*,label=(\arabic*), wide, labelwidth=0pt, labelindent=0pt]
\item By Lemma \ref{lem:difftrick} and \eqref{eq:betaHdiff}, we obtain
	\begin{equation*}
		\frac{\del}{\del\ol\t} \b\left(\frac{Dv^2}{\vt{Q(\t,1)}^2};k+\frac12,\frac12\right) = \frac{iD^{k+\frac12}v^{2k}\sgn(Q_\t)}{\vt{Q(\t,1)}^{2k}Q(\ol\t,1)}.
	\end{equation*}
	This implies the claim.
\item Lemma \ref{lem:bolPsiInduction} implies that
	\[
		\frac12\D^{2k+1} \left(g_k^{[1]}(\t)\right) = -\frac{D^{k+\frac12}(2k)!}{(4\pi)^{2k+1}} \frac{\sgn(Q_\t)}{Q(\t,1)^{k+1}},
	\]
	from which we deduce the claim by \eqref{eq:Lambdadef}. 
\item The claim follows directly from \eqref{eq:split} along with part (1) and \eqref{eq:Lambdadef}. \qedhere
\end{enumerate}
\end{proof}

\subsection{Further properties of \texorpdfstring{$\Psi_{-k,D}$}{\textPsi(-k)D} and the proof of Theorem \ref{thm:Psimain}}

We begin with the local behaviour of $\Psi_{-k,D}$. Similar as in the proof of Proposition \ref{prop:Lambdalocal}, we obtain.
\begin{prop}\label{prop:Psilocal}
	Let $\t\in E_D$.
	\begin{enumerate}[leftmargin=*,label=\rm{(\arabic*)}]
		\item We have
		\[
			\lim_{\eps\to0^+} (\Psi_{-k,D}(\t+i\eps)-\Psi_{-k,D}(\t-i\eps)) = 0.
		\]
	
		\item We have
		\[
			\tfrac12\lim_{\eps\to0^+} (\Psi_{-k,D}(\t+i\eps)+\Psi_{-k,D}(\t-i\eps)) = \Psi_{-k,D}(\t).
		\]
		
		\item We have
		\[
			\lim_{\eps\to0^+} \left(\frac{\del}{\del\ol\tau} \Psi_{-k,D}(\tau+i\eps) - \frac{\del}{\del\ol\tau} \Psi_{-k,D}(\tau-i\eps)\right) = iD^{k+\frac12}v^{2k} \sum_{\substack{Q \in \Qc_D \\ Q_{\tau} = 0}} \frac{\sgn(Q)}{Q\left(\ol{\t},1\right)^{k+1}}.
		\]
	\end{enumerate}
\end{prop}

Furthermore, we have, for every $\tau \in \H \setminus E_D$,
\begin{equation} \label{eq:Eichlerdiff}
\begin{aligned}
&\xi_{-2k} \left(\Lambda_{k+1,D}^{*}(\tau)\right) = \Lambda_{k+1,D}(\tau), \qquad &&\D^{2k+1} \left(\Lambda_{k+1,D}^{*}(\tau)\right) = 0, \\
&\xi_{-2k} \left(\mathcal{E}_{\Lambda_{k+1,D}}(\tau)\right) = 0,  &&\D^{2k+1} \left(\mathcal{E}_{\Lambda_{k+1,D}}(\tau)\right) = \Lambda_{k+1,D}(\tau).
\end{aligned}
\end{equation}
The third claim follows by holomorphicity of $\Ec_{\La_{k+1,D}}$, while the second claim holds as $\Lambda_{k+1,D}^*$ (as a function of $\t$) is a polynomial of degree at most $2k$ by \eqref{eq:Eichlerdef}. The first and fourth claim follow by a standard calculation using the integral representations from \eqref{eq:Eichlerdef} directly.

Now, we are ready to prove Theorem \ref{thm:Psimain}.

\begin{proof}[Proof of Theorem \ref{thm:Psimain}]
	\begin{enumerate}[leftmargin=*,label=(\arabic*), wide, labelwidth=0pt, labelindent=0pt]
	\item[(2)] We define 
	\[
		f(\t) \coloneqq \Psi_{-k,D}(\t) + \frac{D^{k+\frac12}(2k)!}{(4\pi)^{2k+1}} \Ec_{\La_{k+1,D}}(\t) - D^{k+\frac12} \La_{k+1,D}^*(\t).
	\]
	Combining Proposition \ref{prop:Psidiff} with \eqref{eq:Eichlerdiff}, we deduce that 
	\[
		\xi_{-2k}\left(f(\tau)\right) = \D^{2k+1}\left(f(\tau)\right) = 0.
	\]
	Hence, $f$ is a polynomial in $\t$ of degree at most $2k$. By Proposition \ref{prop:Psilocal} (1), $\Psi_{-k,D}$ has no jumps on $E_D$. Thus, we may freely select an arbitrary connected component of $\H \setminus E_D$ to compute $f$. Choosing the connected component of $\H \setminus E_D$ containing $i\infty$, we are in the same situation as in the induction start during the proof of \cite{BKK}*{Theorem 7.1}. In other words, the function $f$ is in fact constant, and this constant was computed in \cite{BKK}*{Lemma 7.3}. We infer that $f$ coincides with $c_\infty$.
	\item[(1)] We verify the four conditions in Definition \ref{defn:LHMF}.
	\begin{enumerate}[leftmargin=*,label=(\roman*), wide, labelwidth=0pt, labelindent=0pt]
	\item Modularity of weight $-2k$ follows by Lemma \ref{lem:InvarFactors} and \eqref{eq:GammaOnQ}.
	\item Local harmonicity with respect to $\De_{-2k}$ outside $E_D$ is shown in Proposition \ref{prop:Psidiff} $(3)$.
	\item The required behaviour on $E_D$ is given in Proposition \ref{prop:Psilocal} (2).
	\item By Theorem \ref{thm:Psimain} (2), $\Psi_{-k,D}$ has most polynomial growth towards $i\infty$. More precise, $\La_{k+1,D}$ admits a Fourier expansion of the shape
	$
	\sum_{n\ge1}c(n)e^{2\pi in\t},
	$
	where the Fourier coefficients depend on the connected component of $\H \setminus E_D$ in which $\tau$ is located and which were computed in \cites{mo22, mo21}. Noting that $c(n) \in \R$ and letting $\Ga(s,y)$ denote the incomplete $\Ga$-function, we obtain, for $v\gg1$,
	\begin{align*}
	\mathcal{E}_{\Lambda_{k+1,D}}(\tau) &= \sum_{n\geq 1} \frac{c(n)}{n^{2k+1}}e^{2\pi i n \tau}, \qquad
	\Lambda_{k+1,D}^*(\tau) = \sum_{n \geq 1} \frac{c(n) }{(4\pi n)^{2k+1}}\Gamma(2k+1,4\pi n v) e^{- 2\pi i n \tau}.
	\end{align*}
	We observe that the holomorphic Eichler integral vanishes as $\t\to i\infty$, and the same holds for the non-holomorphic Eichler integral due to \cite{nist}*{§8.11 (i)}. This proves that
	\[
		\lim_{\t\to i\infty} \Psi_{-k,D}(\tau) = c_\infty.
	\]
	Proposition \ref{prop:Psilocal} (1) yields that the singularities of $\Psi_{-k,D}$ on $E_D$ are continuously removable. Combining Proposition \ref{prop:Psilocal} $(3)$ with Lemmas \ref{lem:difftrick} and \ref{lem:localnonzero} shows that $\Psi_{-k,D}$ has no differentiable continuation to $E_D$. This completes the proof. \qedhere
\end{enumerate}
\end{enumerate}
\end{proof}


\begin{thebibliography}{99}
	\bibitem{nist} R. Boisvert, C. Clark, D. Lozier, and F. Olver, {\it NIST Handbook of Mathematical Functions}, Cambridge University Press (2010).
	
	\bibitem{thebook} K. Bringmann, A. Folsom, K. Ono, and L. Rolen, {\it Harmonic Maass forms and mock modular forms: theory and applications}, American Mathematical Society Colloquium Publications {\bf64}, American Mathematical Society, Providence, RI, (2017).
	
	\bibitem{BKK} K. Bringmann, B. Kane\ and\ W. Kohnen, {\it Locally harmonic Maass forms and the kernel of the Shintani lift}, Int. Math. Res. Not. IMRN (2015), no.~11, 3185--3224.

	\bibitem{BMR} K. Bringmann, A. Mono\ and\ L. Rolen, {\it Flipping operators and locally harmonic Maass forms}, Ramanujan J. {\bf 68} (2025), no.~2, 40.
	
	\bibitem{brufu} J. Bruinier and J. Funke, {\it On two geometric theta lifts}, Duke Math. J. {\bf125} (2004), 45--90.

	\bibitem{bruonrho} J. Bruinier, K. Ono,\ and\ R. Rhoades, {\it Differential operators for harmonic weak Maass forms and the vanishing of Hecke eigenvalues}, Math. Ann. {\bf 342} (2008), 673--693. 

	\bibitem{DIT10} W. Duke, O. Imamo\={g}lu, and\ A. T\'{o}th, {\it Rational period functions and cycle integrals}, Abh. Math. Semin. Univ. Hambg. {\bf 80} (2010), no.~2, 255--264.

	\bibitem{DIT11} W. Duke, O. Imamo\={g}lu, and A. T\'{o}th, {\it Cycle integrals of the $j$-function and mock modular forms}, Ann. of Math. (2) {\bf 173} (2011), 947--981.

	\bibitem{eichler} M. Eichler, {\it Eine Verallgemeinerung der Abelschen Integrale}, Math. Z. {\bf 67} (1957), 267--298.
	
	\bibitem{hoevel} {M. Hövel}, \textit{Automorphe Formen mit Singularitäten auf dem hyperbolischen Raum}, {Ph.D. Thesis}, {TU Darmstadt}, (2012).

	\bibitem{IOS} \"O. Imamo\=glu and C. O'Sullivan, {\it Parabolic, hyperbolic and elliptic Poincaré series}, Acta Arith. {\bf139} (2009), no. 3, 199--228.

	\bibitem{katok} S. Katok, {\it Closed geodesics, periods and arithmetic of modular forms}, Invent. Math. {\bf 80} (1985), no.~3, 469--480.
	
	\bibitem{knopp90} {M. Knopp}, \textit{Modular integrals and their Mellin transforms}, (1990), {327--342}, \textit{Analytic number theory}, {Allerton Park, IL}, {(1989)}, {Progr. Math.} \textbf{85}, {Birkh\"{a}user Boston, Boston, MA}.

	\bibitem{knopp62} M. Knopp, {\it On abelian integrals of the second kind and modular functions}, Amer. J. Math. {\bf 84} (1962), 615--628.

	\bibitem{knopp64} M. Knopp, {\it On generalized abelian integrals of the second kind and modular forms of dimension zero}, Amer. J. Math. {\bf 86} (1964), 430--440.

	\bibitem{knopp74} M. Knopp, {\it Some new results on the Eichler cohomology of automorphic forms}, Bull. Amer. Math. Soc. {\bf 80} (1974), 607--632.

	\bibitem{knopp78} M. Knopp, {\it Rational period functions of the modular group}, Duke Math. J. {\bf45} (1978), 47--62.

	\bibitem{knopp81} M. Knopp, {\it Rational period functions of the modular group. II}, Glasgow Math. J. {\bf 22} (1981), no.~2, 185--197.
	
	\bibitem{koh} W. Kohnen, {\it Fourier coefficients of modular forms of half-integral weight}, Math. Ann. {\bf271} (1985), 237--268.

	\bibitem{koza81} W. Kohnen and D. Zagier, {\it Values of $L$-series of modular forms at the center of the critical strip}, Invent. Math. {\bf64} (1981), 175--198.

	\bibitem{koza84} W. Kohnen and D. Zagier, {\it Modular forms with rational periods}, {\it Modular forms}, Durham, (1983) {\it Ellis Horwood Ser. Math. Appl.: Statist. Oper. Res.}, Horwood, Chichester, (1984), 197--249.

	\bibitem{lehner71} J. Lehner, {\it The Eichler cohomology of a Kleinian group}, Math. Ann. {\bf 192} (1971), 125--143.

	\bibitem{mo22} A. Mono, {\it Eisenstein series of even weight $k\geq2$ and integral binary quadratic forms}, Proc. Amer. Math. Soc. {\bf 150} (2022), no.~5, 1889--1902.
	
	\bibitem{mo21} A. Mono, {\it Locally harmonic Maass forms of positive even weight}, Israel J. Math. {\bf 261} (2024), no.~2, 671--694.
	
	\bibitem{parson} L. Parson, {\it Modular integrals and indefinite binary quadratic forms}, {\it A tribute to Emil Grosswald: number theory and related analysis}, Contemp. Math. {\bf143}, Amer. Math. Soc., Providence, RI (1993), 513--523.
	
	\bibitem{shim} G. Shimura, {\it On modular forms of half integral weight}, Ann. of Math. (2) {\bf97} (1973), 440--481.
	
	\bibitem{shin} T. Shintani, {\it On construction of holomorphic cusp forms of half integral weight}, Nagoya Math. J. {\bf58} (1975), 83--126.

	\bibitem{stza} J. Stienstra and D. Zagier, {\it Bimodular forms and holomorphic anomaly equation}, In: Workshop on Modular Forms and String Duality, Banff International Research Station, 2006, \url{https://www.birs.ca/workshops/2006/06w5041/report06w5041.pdf}
	
	\bibitem{zagier75} D. Zagier, {\it Modular forms associated to real quadratic fields}, Invent. Math. {\bf30} (1975), 1--46.
	
	\bibitem{zagier77} D. Zagier, {\it Modular forms whose Fourier coefficients involve zeta-functions of quadratic fields}, {\it Modular functions of one variable, VI}, Proc. Second Internat. Conf., Univ. Bonn, Bonn (1976), Springer, Berlin, (1977) 105--169. Lecture Notes in Math., Vol. 627.

	\bibitem{zagier81} D. Zagier, {\it Zetafunktionen und quadratische K\"{o}rper}, Hochschultext, Springer-Verlag, Berlin, 1981.
\end{thebibliography}
\end{document}